\documentclass[12pt, twoside]{article}

\usepackage{amsmath, amsfonts, amsthm, amssymb, mathrsfs}

\usepackage{mathtools}
\usepackage{microtype}

\usepackage{float}
\usepackage{enumerate}
\usepackage{graphicx}

\usepackage{multicol}

 \usepackage[all]{xy}

\usepackage[hidelinks]{hyperref}

\usepackage{paralist}
\usepackage{geometry}
\def\Z{\mathbb{Z}}
\def\R{\mathbb{R}}

\def\N{\mathcal{N}}
\def\H{\mathcal{H}}

\newtheorem{theorem}{Theorem}[section]
\newtheorem{lemma}[theorem]{Lemma}
\newtheorem{proposition}[theorem]{Proposition}
\newtheorem{corollary}[theorem]{Corollary}

\theoremstyle{definition}

\newtheorem{definition}[theorem]{Definition}
\newtheorem{remark}[theorem]{Remark}
\newtheorem{example}[theorem]{Example}

\def\int{\operatorname{Int}}
\def\es{\operatorname{S}}
\def\de{\operatorname{D}}
\def\reeb{\mathcal{R}}

\def\ab{\operatorname{Ab}}
\def\geng#1{\langle #1 \rangle}
\def\ind{\operatorname{ind}}

\def\im{\operatorname{Im}}
\def\corank{\operatorname{corank}}
\def\cone{\operatorname{Cone}}
\def\rank{\operatorname{rank}}
\def\id{\operatorname{id}}

\def\epi{\operatorname{Epi}}

\def\aut{\operatorname{Aut}}

\def\diffd{\operatorname{Diff_{\text{\tiny$\bullet$}}}}

\def\rp{\R\!\operatorname{P^2}}
\def\deg{\operatorname{deg}}

\usepackage{tikz}
\usetikzlibrary{matrix,positioning,calc}
\usetikzlibrary{decorations.pathmorphing}
\usetikzlibrary{decorations.pathreplacing}

\tikzset{snake it/.style={-stealth,
        decoration={snake,
            amplitude = .4mm,
            segment length = 2mm,
            post length=0.9mm},decorate}}

\makeatletter
\def\@addpunct#1{%
    \relax\ifhmode
    \ifnum\spacefactor>\@m \else#1\fi
    \fi}
    \newcommand{\zz}[1]{}
\newcommand{\keywordsname}{$2020$ Mathematics Subject Classification}
\def\@setkeywords{%
    {\itshape \keywordsname.}\enspace \@keywords\@addpunct.}
\def\keywords#1{\def\@keywords{#1}}
\let\@keywords=\@empty
\g@addto@macro{\maketitle}{\begingroup%
    \let\@makefnmark\relax  \let\@thefnmark\relax%
    \ifx\@keywords\@mpty\else\@footnotetext{\@setkeywords}\fi%
    \endgroup}
\makeatletter

\keywords{Primary: 20F65; Secondary: 57M15, 57R90.  \\
    \indent\indent{\itshape Key words and phrases}: 2-sided submanifold, framed cobordism, Reeb graph, corank of group, \\\indent\indent equivalence of epimorphisms onto free group.  
   \\\indent\indent The authors were supported by the National Science Centre, Poland within the grants: \\
   \indent\indent NCN 2015/19/B/ST1/01458, and  Sheng~1 UMO-2018/30/Q/ST1/00228.
}

\newcommand{\address}{{ \bigskip

\footnotesize
    {\noindent\textsc{Wac\l{}aw Marzantowicz}\\
        Adam Mickiewicz University, Pozna\'n, Faculty of Mathematics and Computer Science\\
        ul. Uniwersytetu Pozna\'nskiego 4, 61-614 Pozna\'n, Poland} \\
        \textit{E-mail address:} \texttt{marzan@amu.edu.pl}\\

\footnotesize
    {\noindent\textsc{\L{}ukasz Patryk Michalak}\\
        Adam Mickiewicz University, Pozna\'n, Faculty of Mathematics and Computer Science\\
        ul. Uniwersytetu Pozna\'nskiego 4, 61-614 Pozna\'n, Poland} \\
    \textit{E-mail address:} \texttt{lukasz.michalak@amu.edu.pl}

}}

\date{}

\title{Relations between Reeb graphs, systems of hypersurfaces and~epimorphisms onto free~groups}

\author{Wac\l{}aw Marzantowicz, \L{}ukasz Patryk Michalak}

\begin{document}

    \maketitle
    \begin{abstract}   	
    	\noindent We construct a correspondence between epimorphisms $\varphi \colon \pi_1(M) \to F_r$ from the fundamental group of a compact manifold $M$ onto the free group of rank $r$, and systems of $r$ framed non-separating hypersurfaces in $M$, which induces a bijection onto framed cobordism classes of such systems. In consequence, for closed manifolds any such $\varphi$ can be represented by the Reeb epimorphism of a~Morse function $f\colon M \to \mathbb{R}$, i.e. by the epimorphism induced by the quotient map $M \to \mathcal{R}(f)$ onto the Reeb graph of $f$. Applying this construction we discuss the problem of classification up to (strong) equivalence of epimorphisms onto free groups, providing a new purely geometrical-topological proof of the solution of this problem for surface groups.\looseness=-1
    \end{abstract}

\section{Introduction}

The Reeb graph $\reeb(f)$ of a Morse function $f\colon M \to \R$ on a closed manifold $M$, as an invariant of the pair $(M,f)$, is a tool of global analysis attracting more attention recently due to its applications to computer graphics as well as its importance in purely mathematical problems (for more details see \cite{Edelsbrunner, Fabio-Landi, Gelbukh:DCG,Gelbukh:filomat, KMS, Prishlyak_nonoriented, Michalak-TMNA, Sharko-functions}). It is constructed by contracting the connected components of levels sets of the function $f$. Since it is a~finite graph, its fundamental group is a~free group  $F_r$ of a finite rank $r \geq 0$. This work is motivated by a natural question: is any epimorphism $\pi_1(M) \to F_r$ represented as the canonical epimorphism $q_\# \colon \pi_1(M)\to \pi_1(\reeb(f))$, induced by the quotient map ${q\colon M \to \reeb(f)}$ for a~Morse function~$f$? The epimorphism $q_\#$ is called the Reeb epimorphism~of~$f$. We give an affirmative answer to this question in Theorem \ref{theorem:realization_of_Reeb_epi}. Below we summarize the main results obtained in this work and the methods used in order to answer the question.


One of the main ingredients in the proof is the correspondence, given by an extended Pontryagin--Thom construction, between homomorphisms $\pi_1(W)\to F_r$ and systems of hypersurfaces  $\mathcal{N}=(N_1,\, \dots\,,N_r)$ consisting of framed and properly embedded submanifolds $N_i$ of codimension $1$ in a compact manifold $W$, possibly with boundary. A~system $\N$ is independent if it is non-separating, and it is regular if each $N_i$ is connected. It is an easy observation that an independent system of hypersurfaces induces a surjective homomorphism. The converse of this fact is the first substantial result of this work (Theorem \ref{theorem:epimorphism_is_induced_by_regular_and_independent_system}). It provides for any epimorphism $\varphi \colon \pi_1(W) \to F_r$ the construction of~a~regular and independent system of hypersurfaces which induces $\varphi$. 



Having these geometric tools, we study the problem of classification of epimorphisms $G \to F_r$ up to equivalence and strong equivalence defined in \cite{Grigorchuk-Kurchanov,Grigorchuk-algebra, Grigorchuk-Zieschang}. Briefly, on the set $\operatorname{Hom}(G,F_r)$ of homomorphisms there are the natural actions of the automorphism groups $\aut(G)$ and $\aut(F_r)$ given by composition. Two homomorphisms are strongly equivalent (resp. equivalent) if they are in the same orbit of the action of $\aut(G)$ (resp. $\aut(G) \times \aut(F_r)$).
First, note that two systems induce the same homomorphism if and only if they are framed cobordant as systems of hypersurfaces (see Definition \ref{def:framed_cobordism_of_systems}).
This leads to a correspondence between strong equivalence classes of epimorphisms $\pi_1(M) \to F_r$ and elements of $\H^{fr}_r(M)/_{\diffd(M)}$, the set of framed cobordism classes of independent and regular systems of size $r$ in $M$ up to diffeomorphisms which preserve the basepoint. It is a one-to-one correspondence if the natural homomorphism $\diffd(M) \to \aut(\pi_1(M))$ is surjective. For example, this holds when $M$ is a closed surface (by the Dehn--Nielsen Theorem) or when $M$ is a~hyperbolic manifold of dimension at least~$3$ (by the Mostow Rigidity Theorem). As an application of the methods we developed, in the proof of Theorem \ref{theorem:calculations_of_cobordisms_for_surfaces} we determine the elements of $\H^{fr}_r(\Sigma)/_{\diffd(\Sigma)}$ for a~closed surface $\Sigma$. This provides a classification up to strong equivalence of epimorphisms $\pi_1(\Sigma) \to F_r$, which was originally shown by R. Grigorchuk, P.~Kurchanov and H.~Zieschang \cite{Grigorchuk-Kurchanov,Grigorchuk-algebra, Grigorchuk-Zieschang} by using more algebraic, but also topological methods (see Theorem~\ref{thm:grigorchuk}).\looseness=-1

Transition from strong equivalence classes to equivalence classes is obtained by considering the action of $\aut(F_r)$, which is generated by elementary Nielsen transformations. We define analogous operations on $\H^{fr}_r(M)$ which cause the same change of an inducing epimorphism as its composition with the corresponding Nielsen transformation. These operations allow us to compute equivalence classes of epimorphisms $\pi_1(\Sigma) \to F_r$ (see~Theorem \ref{theorem:equivalence_for_surfaces_Nielsen_transformations}) as in the Grigorchuk--Kurchanov--Zieschang Theorem.


Next, we exhibit relations to Reeb graph theory. Extending the methods of the second author from~\cite{Michalak-DCG} we assign in Theorem \ref{theorem:factorization_by_Reeb_epimorphism} a~Morse function $f$ on $W$ and its Reeb graph to any system of hypersurfaces without boundary in such a way that its induced homomorphism is factorized by the Reeb epimorphism of $f$. Moreover, if the system is independent, this gives the construction of the initial graph (see Figure \ref{figure:initial_graph}) as the Reeb graph such that submanifolds from the system are components of the same level set of $f$. Subsequently, one of the main result of the paper, Theorem \ref{theorem:realization_of_graph_for_surfaces_with_prescribed_level_sets}, provides, for a regular and independent system $\N$ of hypersurfaces and a graph $\Gamma$ with natural necessary conditions, the construction of a Morse function realizing $\Gamma$ as its Reeb graph such that submanifolds from $\N$ correspond to prescribed edges of $\Gamma$ outside a spanning tree. 

Prescribed components of level sets are an additional ingredient to the realization theorems for Reeb graphs. The classical result of V. Sharko \cite{Sharko} provides a realization of any graph with the so-called good orientation as the Reeb graph of a function on a surface. Recently, the second author \cite{Michalak-TMNA,Michalak-DCG} resolved the realization problem with an arbitrary fixed closed manifold. In the case of surfaces the realization is done up to isomorphism of graphs with detailed description of Reeb graphs of Morse functions. For higher-dimensional manifolds it is up to homeomorphism and the construction relies on combinatorial modifications of Reeb graphs. It was showed that any graph $\Gamma$ with good orientation is obtained from the initial graph by using a finite sequence of combinatorial modifications. In this work, we extend these results to the situation, when the manifold $W$ has a boundary and one can prescribe connected components of level sets of the function corresponding to edges of the graph outside a spanning tree.

The principal significance of Theorem \ref{theorem:realization_of_graph_for_surfaces_with_prescribed_level_sets} is that it allows one to represent any epimorphism $\varphi\colon \pi_1(M)\to \pi_1(\Gamma)$ as the Reeb epimorphism of a Morse function whose Reeb graph is homeomorphic to $\Gamma$ (Corollary \ref{corollary:Reeb_epi_for_closed_manifolds}). Theorems \ref{theorem:Reeb_epi_iff_regular_and_independent} and \ref{theorem:realization_of_Reeb_epi} provide an answer to the initial question for a manifold $W$ with boundary. An epimorphism $\pi_1(W)\to \pi_1(\Gamma)$ is represented as the Reeb epimorphism if and only if it is induced by a system of hypersurfaces without boundary.
Equivalently, it is factorized through $\pi_1(W)/\langle\pi_1(\partial W)\rangle^{\pi_1(W)}$, where $\langle\pi_1(\partial W)\rangle^{\pi_1(W)}$ is the smallest normal subgroup of $\pi_1(W)$ containing the classes of all loops from $\partial W$.

Note that the problem of representability of an epimorphism as the Reeb epimorphism was also considered independently by O. Saeki \cite{Saeki_Reeb_spaces}; for a finite graph $\Gamma$ without loops and a closed manifold he constructs a smooth function with finitely many critical values such that its Reeb graph is isomorphic to $\Gamma$ and under this identification its Reeb epimorphism is $\varphi$. Thus Saeki realizes not only the topological structure of $\Gamma$, but also the combinatorial one, at the cost of losing the non-degeneracy of critical points. Note that the number of vertices of degree $2$ in the Reeb graph of a Morse function cannot be arbitrary (see, for instance \cite[Theorem 5.6]{Michalak-TMNA}), and thus we focus on the homeomorphism type. Our results are also different in that way they deal with manifolds with boundary and allow us to control the system of hypersurfaces in connected components of level sets of the constructed function.

Another subject of this paper are the maximum values of some related quantities. The corank of a finitely generated group $G$
is the maximum rank $r$ for which there exists an epimorphism $G\to F_r$. As defined in \cite{Michalak-TMNA} for closed manifolds,
the Reeb number $\reeb(W)$ of $W$ is the maximum cycle rank of Reeb graphs of Morse functions $f\colon W\to \R$ which are constant on each connected component of $\partial W$.
In other words, $\reeb(W)$ is the maximum rank of the Reeb epimorphism of such a Morse function on $W$. For closed manifolds we have the equality $\reeb(M) = \corank (\pi_1(M))$ (see \cite{Michalak-DCG} and~\cite{Gelbukh:filomat}).
In Theorem \ref{theorem:Reeb_epi_iff_regular_and_independent} we establish  the corresponding formula for manifolds with boundary:
$$
\reeb(W) = \corank\left(\pi_1(W)/\langle\pi_1(\partial W)\rangle^{\pi_1(W)}\right)\,.
$$
It is also equal to the maximum size of an independent system of hypersurfaces without boundary in $W$. The last quantity we consider is the maximum size of an independent system of hypersurfaces in $W$, which was denoted by $C(W)$ by O. Cornea in~\cite{Cornea}. It is always equal to the corank of $\pi_1(W)$.

Relations between these numbers have already been studied by other authors. The equality $C(W) =\corank(\pi_1(W))$ was established by O. Cornea \cite{Cornea} for closed smooth manifolds and by W.~Jaco~\cite{Jaco} for combinatorial manifolds with boundary.  The equality $R(M)=\corank(\pi_1(M))$ was proved by the second author \cite{Michalak-DCG} and independently by I.~Gelbukh \cite{Gelbukh:DCG} for orientable manifolds by using foliation theory and later in \cite{Gelbukh:filomat} without the orientability assumption by other methods. It is worth emphasizing that, while these papers contain geometric descriptions of the corank of $\pi_1(M)$, no~correspondence between epimorphisms, systems of hypersurfaces and Reeb graphs was given. This work fills this gap. \looseness=-1

The paper is organized as follows. 
In Section \ref{section:systems_of_hypersurfaces} we describe the correspondence  between systems of hypersurfaces and homomorphisms onto free groups.
Next, in Section \ref{section:cobordism_of_systems_and_classification_of_epimorphism} we deal with the problem of classification of epimorphisms onto free groups up to equivalence and strong equivalences.
Section \ref{section:reeb_epimorphisms} establishes the
representation of epimorphisms onto free groups as the Reeb epimorphisms of Morse functions. The rest of the section is devoted to its some applications containing description of the corank and Reeb number, the problem of extendability of independent systems and connections with topological conjugacy of functions. \looseness=-1


\section{Systems of hypersurfaces and induced homomorphisms}\label{section:systems_of_hypersurfaces}

Throughout the paper we assume that all manifolds are
smooth of dimension $n \geq 2$. 
Hereafter, $M$ and $W$ are connected and compact smooth manifolds with fixed basepoints and $M$ is closed, unless otherwise stated.

We  use  the following model  of $F_r$, the free group on $r$
generators. Consider the circle $\es^1$ as the quotient
$[-1,1]/\{-1,1\}$ and take $F_r := \pi_1(\bigvee_{i=1}^r \es^1_i)$
as the fundamental group of the wedge product of $r \geq 1$
copies of the circle. We use the convention that $\bigvee^0_{i=1}
\es^1_i = {\rm pt}$, thus  $F_0 = 1$ is the trivial group. 

We will omit basepoints from the notation.

\subsection{Systems of hypersurfaces}

Let $W$ be a compact manifold. A submanifold $N$ of $W$ is called \emph{proper} if $N \cap \partial W = \partial N$.
A \emph{framing} of a submanifold $N$ in $W$ is a smooth function $\nu$ which assigns to each $x\in N$ a basis of the normal bundle of $N$ at the point $x$. The pair $(N,\nu)$ is called a \emph{framed submanifold}.
If $N$ is of codimension~1, then its framing is just a nonzero section of the normal bundle of $N$. Thus $N$ has a closed product neighbourhood $P(N) \cong N \times [-1,1]$ and it is called \emph{$2$-sided}. We assume that $P(N)$ is compatible with the framing. Denote by $P_t(N)$ the submanifold corresponding to $N \times \{t\}$. The positive side of $N$ containing $P_t(N)$ for $t\in(0,1]$ agrees with the side determined by the framing.

A \emph{system of hypersurfaces} in $W$ is a tuple $\mathcal{N}=(N_1,\ldots,N_r)$ of disjoint, proper, $2$-sided submanifolds $N_i$ together with their framings $\nu_i$.
The number $r$ is called the \emph{size} of     the system $\mathcal{N}$.
Denote by
$$
W | \N := W \setminus \bigcup_{i=1}^r \int P(N_i)
$$
the complement of the system $\N$ for sufficiently small product neighbourhoods of $N_i$'s. It will cause no confusion if we use $\N$ to designate also $\bigcup_{i=1}^r N_i$, the sum of all submanifolds from the system. Of course, framings $\nu_i$ of submanifolds $N_i$ form a framing $\nu$ of $\N$ such that $\nu|_{N_i} = \nu_i$. Unless it is necessary, we will not write a framing of
a system explicitly.

A system $\mathcal{N}$ is called \emph{independent} if $W | \N$ is connected, and it is called \emph{regular} if each $N_i$ is connected. The system $\mathcal{N}$ is \emph{without boundary} if
$\partial \N = \varnothing$. Note that we do not assume that submanifolds $N_i$ are connected, unless $\N$ is regular.

Now we define the extended Pontryagin--Thom construction for a system of hypersurfaces.

\begin{definition}\label{definition:induced by system of hyperspaces}
	The \emph{homomorphism} $\varphi_\N \colon \pi_1(W) \to F_r$ \emph{induced by a system} $\mathcal{N} =(N_1,\ldots, N_r)$ omitting the basepoint is defined as follows. 
	Fix product neighbourhoods $P(N_i)\cong N_i \times [-1,1]$ which are disjoint.
	We define the map $f_\N \colon W \to \bigvee_{i=1}^r \es^1_i$ which maps $W |\N$ to the basepoint and each $P(N_i)$ onto $i$-th circle $\es^1_i = [-1,1]/\{-1,1\}$ by mapping $P_t(N_i)$ into $t$. It is clear that $f_\N$ is continuous, thus we put $\varphi_\N := (f_\N)_\#$ to be the homomorphism induced by $f_\N$ on fundamental groups.
\end{definition}

By the definition of a system of hypersurfaces $\varphi_\N$ is well-defined and it is clear that it does not depend on the choice of $P(N_i)$'s and a given framing, but on the orientation
of the normal bundle of $\N$.

\begin{proposition} \label{proposition:homomorphisms_are_induced_by_systems}
	Any homomorphism $\varphi \colon \pi_1(W) \to F_r$ is induced by a system of hypersurfaces.  If a system $\mathcal{N}$ is independent, then $\varphi_\N$ is an epimorphism.
\end{proposition}

\begin{proof}
	Since $\bigvee_{i=1}^r \es^1_i$ is an Eilenberg--MacLane space K($F_r$,1), there is a map $f\colon W \to \bigvee_{i=1}^r \es^1_i$ such that $f_\# = \varphi$.
	Smooth it outside the inverse image of basepoint and take regular values $a_i \in \es^1_i$ of both $f$ and $f|_{\partial W}$. Since $W$ is compact, there is a neighbourhood
	$[a_i -\varepsilon, a_i + \varepsilon]$ consisting of regular values, and thus $N_i := f^{-1}(a_i)$ is a $2$-sided, proper submanifold with product neighbourhood
	$f^{-1}([a_i -\varepsilon, a_i + \varepsilon]) \cong N_i \times [a_i -\varepsilon, a_i + \varepsilon]$. Take the map $h \colon \bigvee_{i=1}^r \es^1_i \to \bigvee_{i=1}^r \es^1_i$ which
	contracts $$\bigvee_{i=1}^r \es^1_i \setminus \bigcup_{i=1}^r [a_i -\varepsilon, a_i + \varepsilon]$$ to the basepoint and maps linearly $[a_i -\varepsilon, a_i + \varepsilon]$ onto $\es^1_i$, preserving orientation. It is clear that $(h\circ f)_\# = \varphi$ is induced by $\N = (N_1,\ldots,N_r)$ with framings compatible with the orientations
	of $[a_i -\varepsilon, a_i + \varepsilon]$.

	If $\N$ is independent, then for any $i$ there is a loop $\alpha_i$ in $(W | \N) \cup P(N_i)$ such that $f_\N \circ \alpha_i$ represents the generator of $\pi_1(\bigvee \es^1_i)$ corresponding to
	$\es^1_i$. Thus $\varphi_\N$ is surjective.
\end{proof}

There is a quite easy characterization, using a special notion of framed cobordism, of systems in a~closed manifold $M$ which induce the same homomorphism to a free group.

Recall (cf. \cite{Milnor_viewpoint}) that submanifolds $N$ and $N'$ in $M$ are \emph{cobordant} if there exists a proper compact submanifold  $W \subset M \times [0,1]$, called \emph{cobordism} between $N$ and $N'$,
such that $W \cap \left(M \times [0,\varepsilon]\right)=  N \times  [0,\varepsilon]$ and $W \cap \left(M \times [1-\varepsilon,1]\right) = N' \times [1-\varepsilon,1]$.
Framed submanifolds $(N,\nu)$ and $(N',\nu')$ are \emph{framed cobordant}, if there is a cobordism $W\subset M\times [0,1]$ between $N$ and $N'$ with a framing $\vartheta$ such
that $\vartheta(x,t) = (\nu(x),0)$ for $(x,t) \in N \times  [0,\varepsilon]$ and $\vartheta(x,t) = (\nu'(x),0)$ for $(x,t) \in N' \times  [1-\varepsilon,1]$.

\begin{definition} \label{def:framed_cobordism_of_systems}
	Let $\mathcal{N} = (N_1,\ldots,N_r)$ and
	$\N'=(N'_1,\ldots,N'_r)$ be two systems in $M$ of the same size~$r$. We say that $\N$ and $\N'$ are \emph{framed cobordant} (as systems of hypersurfaces) if there are $r$ disjoint framed
	cobordisms $W_i \subset M \times[0,1]$ between $N_i$ and $N'_i$.
	
	In other words, the systems $\mathcal{N}$ and $\mathcal{N}'$ are framed
	cobordant, if framed submanifolds $\N$ and $\N'$
	are framed cobordant by the cobordism $W$ which has a partition into $r$
	parts $W = W_1 \sqcup \ldots \sqcup W_r$ such that $\partial W_i = N_i \times \{0\} \sqcup
	N'_i \times \{1\}$. Clearly, it is an equivalence relation in the
	family of systems of hypersurfaces in $M$ of size $r$. Note that the cobordisms $W_i$ form the system $\mathcal{W}=(W_1,\ldots,W_r)$ of hypersurfaces in $M\times[0,1]$.
\end{definition}

Note that the notion of framed cobordism between systems of hypersurfaces of size 1 is the same as an ordinary framed cobordism.

\begin{proposition}\label{proposition:framed_cobordant_systems=the_same_induced_epimorphisms}
	Systems $\mathcal{N}$ and $\mathcal{N}'$ of hypersurfaces in $M$ are
	framed cobordant if and only if $\varphi_\N = \varphi_{\N'}$.
\end{proposition}

\begin{proof}
	If $\N$ and $\N'$ are framed cobordant by framed cobordisms $W_1,\ldots,W_r$ which form the system $\mathcal{W}$, then as in Definition \ref{definition:induced by system of hyperspaces} it leads to the map $f_\mathcal{W} \colon M\times[0,1] \to \bigvee_{i=1}^r \es^1_i$ for a fixed product neighbourhood
	$P(\mathcal{W})$. It is clear that $f_{\mathcal{W}} |_{M\times \{0\}} = f_\N$ and $f_{\mathcal{W}} |_{M\times \{1\}} = f_{\N'}$ for product neighbourhoods $P(\N) = P(\mathcal{W}) \cap M\times \{0\}$
	and $P(\N') = P(\mathcal{W}) \cap M\times \{1\}$, respectively. Thus $f_{\mathcal{W}}$ is a homotopy between $f_\N$ and $f_{\N'}$, so $\varphi_\N = \varphi_{\N'}$.
	
	Conversely, if $\varphi_\N = \varphi_{\N'}$, then $f_\N $ and $f_{\N'}$ are homotopic by a map $f \colon M \times [0,1] \to \bigvee_{i=1}^r \es^1_i$ which is smooth outside the preimage of basepoint since $\bigvee_{i=1}^r \es^1_i$ is
	an Eilenberg--MacLane space K($F_r$,1).  As in the proof of Proposition \ref{proposition:homomorphisms_are_induced_by_systems} take regular values $a_i \in \es^1_i$ and framed submanifolds
	$W_i = f^{-1}(a_i)$ which form a~system of hypersurfaces in $M\times[0,1]$. They are framed cobordisms between $f_\N^{-1}(a_i) \cong N_i$ and $f_{\N'}^{-1}(a_i) \cong N'_i$. By the construction of
	$f_\N$ and $f_{\N'}$ it is clear that the system $(f_\N^{-1}(a_1),\ldots,f_\N^{-1}(a_r))$ is framed cobordant to~$\N$ and $(f_{\N'}^{-1}(a_1),\ldots,f_{\N'}^{-1}(a_r))$ is framed cobordant to~$\N'$.
	The statement follows by transitivity of framed cobordism.
\end{proof}

\begin{remark}
	It is easy to check that if two systems of hypersurfaces differ only in their framings, but the determined positive sides are the same, then they are framed cobordant. Thus the induced homomorphism depends only on the choice of sides of submanifolds from a system, not on particular framings.
\end{remark}

\subsection{Epimorphisms and independency of inducing systems}

The aim of this section is to prove that any epimorphism onto a free group is induced by an independent and regular system.

Let $\mathcal{N}=(N_1,\ldots,N_r)$ be a system of hypersurfaces in a compact and connected manifold $W$. Note that any class of loops in $W$ can be represented by a loop in the interior $\int W$.

\begin{lemma}\label{lemma:cutting_loops_into_arcs}
	Any class of loops $\omega \in \pi_1(W)$ either can be represented by a loop in $W|\N$ or there is a loop $\alpha \in \omega$ which can be written as the concatenation of
	paths $\alpha_{1} \cdot \ldots \cdot \alpha_{k}$ whose ends lie in $W | \N$ and $\alpha_i \cap P_t(\N)$ is a single point for any $t\in[-1,1]$.
	Thus putting $a_i := [\es^1_i]$ as the generators of $F_r = \pi_1\left(\bigvee_{i=1}^r \es^1_i\right)$ we have $\varphi_\N(\omega) = a^{\epsilon_1}_{i_1}\ldots a^{\epsilon_k}_{i_k}$,
	where $\epsilon_j \in \{-1,+1\}$ and $i_j$ is a unique index for which $\alpha_j \cap N_{i_j}$ is non-empty.
\end{lemma}

\begin{proof}
	Take any loop in $\omega$ and homotope it to be in general position to $\N$. Since they have complementary dimensions, their intersection is a finite set. Now, cut the obtained loop into paths $\alpha_i$ as it is required.
\end{proof}

\begin{lemma}\label{lemma:connected_sum_of_submanifolds}
	Suppose there is a path $\gamma\colon [0,1] \to W$ such that $\gamma \cap \N = \gamma \cap N_j = \{x,y\}$, where $x = \gamma(0) \in X$ and $y=\gamma(1)\in Y$ are in the different connected components $X$ and $Y$ of $N_j$, and which joins $x$ and $y$ from the same side, i.e. $\gamma \cap P_t(N_j) = \varnothing$ for any $t \in [-1,0)$ or for any $t\in(0,1]$. Then there is a system $\N' = (N'_1,\ldots,N'_r)$ such that $N_i = N'_i$ for $i\neq j$, $N'_j$ has a one less connected component than $N_j$ and $\varphi_\N = \varphi_{\N'}$.
\end{lemma}

\begin{proof}
	First, change $\gamma \colon [0,1] \to W$ to be an embedded arc in $\int W$ with the same properties as in the statement. Take a small, closed tubular neighbourhood $P(\gamma)$ of $\gamma$
	parametrized by $\gamma \times \de_3^{n-1}$ such that $P(\gamma) \cap \N = P(\gamma) \cap N_j$, where $\de_t^{n-1} = \{ x \in \R^{n-1} \ : \ ||x|| \leq t\}$ is a closed disc of radius $t$.
	We may assume that $P(\gamma) \cap X = \{x\} \times \de_3^{n-1}$ and $P(\gamma) \cap Y = \{y\} \times \de_3^{n-1}$.
	Now, perform the connected sum operation of $X$ and $Y$ along $\gamma$ in $W$, i.e. we define the new submanifold
	$$
	A = X \#_{\gamma} Y := \left(X \setminus \left(\{x\} \times \de_2^{n-1}\right)\right) \cup  \left(\gamma \times \partial\de_2^{n-1}\right) \cup \left(Y \setminus \left(\{y\} \times \de_2^{n-1}\right)\right).
	$$
	Obviously, $A$ is a topological manifold, smoothly embedded outside $\{x,y\} \times \partial\de_2^{n-1}$. Thus take an open $\varepsilon$-neighbourhood $U$ of
	$\{x,y\} \times \partial\de_2^{n-1}$ and smooth the corners inside $U$. Hence we may assume that $A$ is a $2$-sided smooth submanifold of $W$ with product neighbourhood $P(A)$ such that
	\begin{equation*}
		\begin{aligned}[t]
			P(A\setminus U) = P\left(X \cup Y \setminus \left( \{x,y\} \times \de_2^{n-1} \right) \setminus U\right) 
			\cup \left(\gamma([\varepsilon,1-\varepsilon]) \times (\de^{n-1}_3 \setminus \int\de^{n-1}_1)\right).
		\end{aligned}
	\end{equation*}
	Since $\gamma$ joins $X$ and $Y$ from the same side, the orientations of their normal bundles induce the orientation of $P(A)$, and thus a framing of $A$.
	
	Let $\mathcal{N}' = (N'_1,\ldots,N'_r)$ be a system such that $N_i = N'_i$ for $i\neq j$ and $N'_j = \left(N_j \setminus (X \cup Y)\right) \cup A$.  We will show that $\varphi_\N = \varphi_{\N'}$.
	Let $[\alpha] \in \pi_1(W)$ be any class of loops in $W$ with the basepoint outside $P(\N)$ and $P(\gamma)$. We may assume that $\alpha$ does not intersect $\left(\{x,y\} \times \de_2^{n-1}\right) \cup\, U$ and it is in general position to $\N'$. Write $\alpha = \alpha_{1} \cdot \ldots \cdot \alpha_{k}$ as in Lemma \ref{lemma:cutting_loops_into_arcs} with respect to the system $\N'$,
	so $\varphi_{\N'}([\alpha]) = a^{\epsilon_1}_{i_1}\ldots a^{\epsilon_k}_{i_k}$. Note that $\varphi_\N([\alpha])$ is obtained from $\varphi_{\N'}([\alpha]) = a^{\epsilon_1}_{i_1}\ldots a^{\epsilon_k}_{i_k}$
	by removing these $a^{\epsilon_j}_{i_j}$ which correspond to $\alpha_j$ such that $\alpha_j \cap \left(\gamma \times \partial\de_2^{n-1} \right)\neq \varnothing$. However, if $\alpha_j$
	intersects $\gamma \times \partial\de_2^{n-1}$ and goes inside $\gamma \times \de_2^{n-1}$ (i.e. it has the end point in $\gamma \times \de_2^{n-1}$), then $\alpha_{j+1}$
	also intersects $\gamma \times \partial\de_2^{n-1}$, since it needs to leave $\gamma \times \de_2^{n-1}$ and does not intersect $\{x,y\} \times \de_2^{n-1}$. Thus $a_{i_j} = a_{i_{j+1}}$
	and $\epsilon_{j+1} = -\epsilon_j$. Therefore $\varphi_\N([\alpha]) = \varphi_{\N'}([\alpha])$, so $\varphi_\N = \varphi_{\N'}$.
\end{proof}

\begin{figure}[h]
	\centering
	
	\begin{tikzpicture}[scale=1]
		\node at (0,0) {\includegraphics[width=215pt]{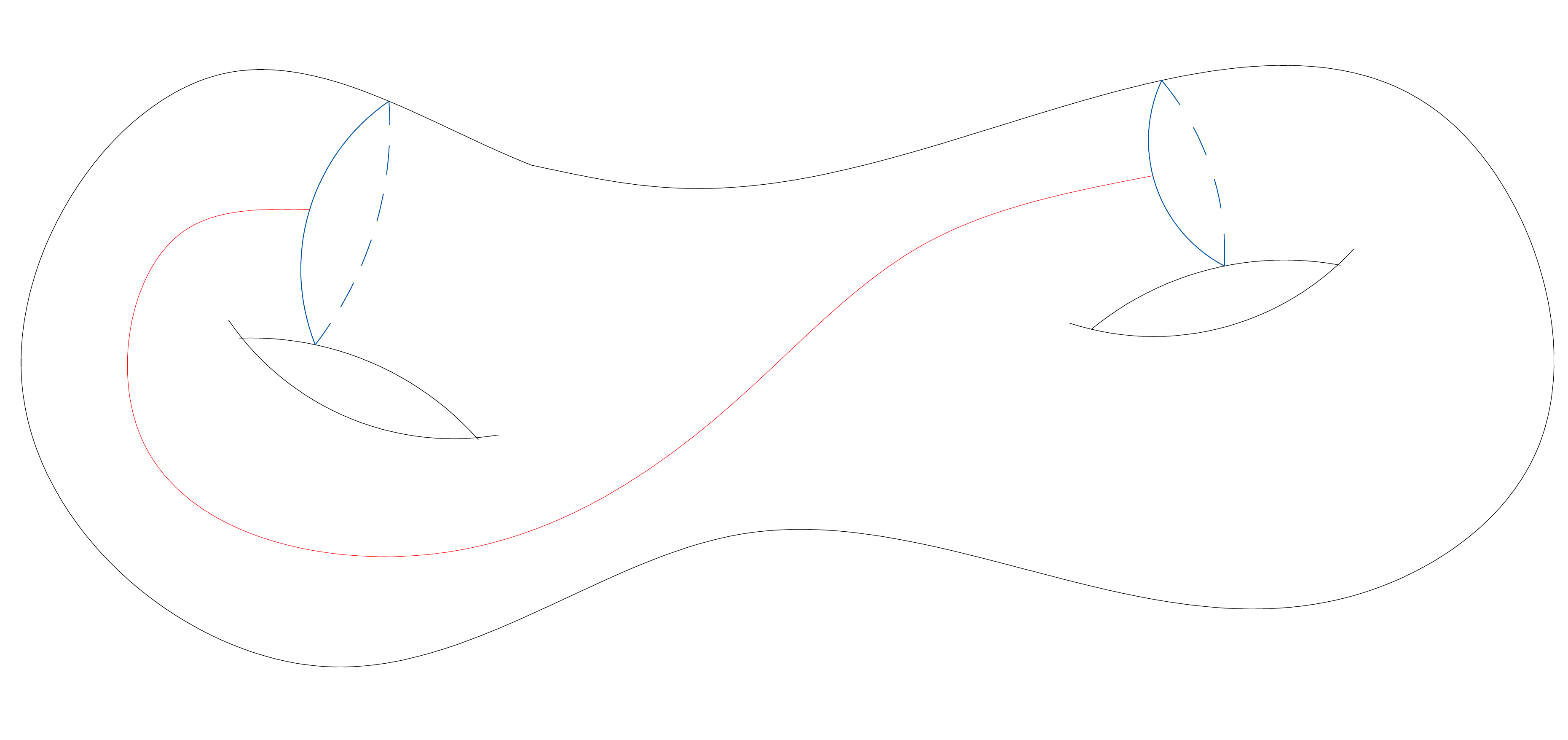}};
		
		\draw (-1.7,0.52) node {$X$};
		\draw (2.4,0.95) node {$Y$}; 
		\draw (0,-0.3) node {$\gamma$}; 
		
		\draw [->] (-2.23,0.9) to (-2.68,1.05); 
		\draw [->] (1.76,1.1) to (1.3,1);

		\draw (0,-1.8) node{(a)};
		

		\node at (8.3,0) {\includegraphics[width=215pt]{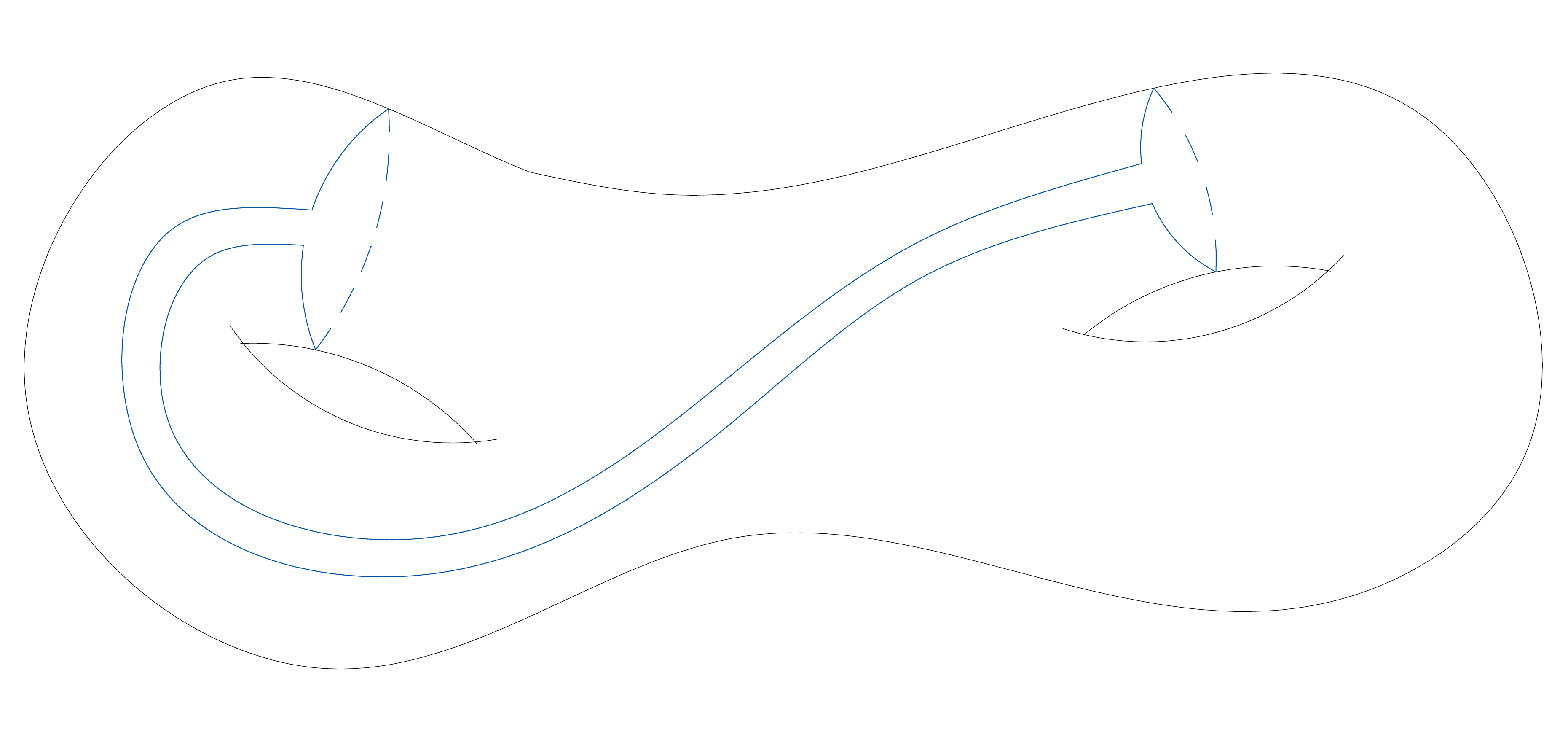}};

		\draw [->] (6.07,0.9) to (5.62,1.05); 
		\draw [->] (10.02,1.1) to (9.6,1); 
		\draw [->] (9.03,0.49) to (9.27,0.17); 
		\draw [->] (7.82,-0.2) to (7.55,0.1); 	
		
		\draw (8.85,-0.3) node {$X\#_\gamma Y$};

		\draw (8.3,-1.8) node{(b)};

	\end{tikzpicture}

	\caption{Example of connected sum operation of submanifolds $X$ and $Y$ along a curve $\gamma$ which joins them from the same side. The arrows indicate	normal vectors from the framing. }\label{figure:connected_sum}
\end{figure}
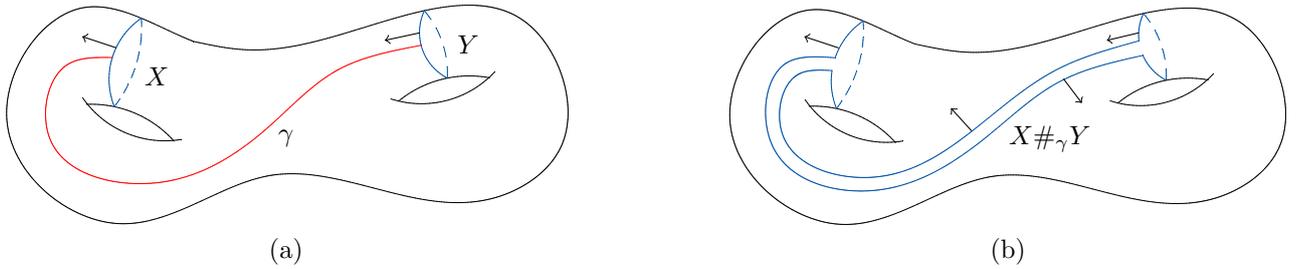

We call the constructed submanifold $X \#_\gamma Y$ the \emph{connected sum} of $X$ and $Y$ \emph{along}~$\gamma$.

\begin{proposition}\label{proposition:epi_and_regular_is_independent}
	Let $\mathcal{N}=(N_1,\ldots,N_r)$ be a system of hypersurfaces in $W$ such that $\varphi_\N$ is an epimorphism and there are no paths as in the statement of Lemma \ref{lemma:connected_sum_of_submanifolds}. Then there is a unique independent and regular system $\mathcal{A}=(A_1,\ldots,A_r)$ in $W$ such that for each $j$ the submanifold $A_j$ is a component of $N_j$ and there is a loop $\tau_j$ such that $\tau_j \cap \mathcal{N} = \tau_j \cap A_j$ is a single point. In particular, if $\N$ is regular, then it is independent.
\end{proposition}

\begin{proof}
	Since $\varphi_\N$ is an epimorphism, for any $j$ there is a loop $\tau_j$ in $W$ such that $f_\N \circ \tau_j$ represents the generator of $F_r = \pi_1\left(\bigvee_{i=1}^r \es^1_i\right)$
	which corresponds to $\es^1_j$. As in Lemma \ref{lemma:cutting_loops_into_arcs} we may consider $\tau_j$ as the concatenation of paths $\alpha^j_{1},\ldots,\alpha^j_{k}$
	such that $a_j = \varphi_\N([\tau_j]) = a^{\epsilon_1}_{i_1}\ldots a^{\epsilon_k}_{i_k}$, where $a_i =[\es^1_i]$. If $k > 1$, then there is some cancellation in the word $a^{\epsilon_1}_{i_1}\ldots a^{\epsilon_k}_{i_k}$, so for some $l$ both $\alpha^j_l$ and $\alpha^j_{l+1}$ intersect the same submanifold $N_{i_l}$. If they intersect two different components of $N_{i_l}$, then it leads to a path as in the statement of Lemma~\ref{lemma:connected_sum_of_submanifolds}, a~contradiction. However, if they intersect $N_{i_l}$ in the same connected component $X$, then we may assume that the starting point of $\alpha^j_l$ and the endpoint of $\alpha^j_{l+1}$ are in $P_t(X)$ for some $t\in[-2,2]\setminus[-1,1]$ by reparameterizing $P(N_{i_l})$. Since $X$ is connected, we may replace the paths $\alpha^j_l$ and $\alpha^j_{l+1}$ in $\tau_j$ by an arc in $P_t(X)$ joining these two points, which provides 
	a loop for which the number of paths in its representation from Lemma \ref{lemma:cutting_loops_into_arcs} is reduced.
	Proceeding inductively we may assume that $\tau_j \cap \N = \tau_j \cap N_j$ is a single point.
	
	Note that if there were two components $X$ and $Y$ of $N_j$ with loops $\tau_X$ and $\tau_Y$ with the same basepoint intersecting $\N$ only in single points of $X$ and $Y$, respectively, then they would determine a path joining $X$ and $Y$ as in Lemma \ref{lemma:connected_sum_of_submanifolds}. Thus for any $j$ there is a unique connected component $A_j$ of $N_j$ with this property.
	
	The system $\mathcal{A} = (A_1,\ldots,A_r)$ is regular by definition and independent by the above property of $A_j$'s. The~uniqueness of $\mathcal{A}$ follows by the uniqueness of its components.
	
	If $\N$ is regular, then $\N=\mathcal{A}$, so it is independent.
\end{proof}

\begin{remark}
	Using the techniques as in the paper of Cornea \cite{Cornea} one can show that for a closed manifold $M$ if $\mathcal{N}$ is not regular and $\varphi_\N$ is surjective,
	then there is an independent and regular system $\mathcal{N}'=(N'_1,\ldots,N'_r)$ in $M$ such that $\N'\subset \N$, without the assumption on the existence of paths as in Lemma \ref{lemma:connected_sum_of_submanifolds}.
\end{remark}

\begin{theorem}\label{theorem:epimorphism_is_induced_by_regular_and_independent_system}
	Any epimorphism $\varphi \colon \pi_1(W) \to F_r$ is induced by a regular and independent system of hypersurfaces.
\end{theorem}

\begin{proof}
	
	Let $\mathcal{N}=(N_1,\ldots,N_r)$ be a system inducing $\varphi = \varphi_\N$ given by Proposition \ref{proposition:homomorphisms_are_induced_by_systems}. By Lemma \ref{lemma:connected_sum_of_submanifolds} we may assume that there is no path as in the statement of the lemma. Thus by Proposition \ref{proposition:epi_and_regular_is_independent}
	we consider a regular and independent system $\mathcal{A} = (A_1,\ldots,A_r)$ such that $A_j$ is a component of $N_j$ and for each $j$ there is a loop $\tau_j$ such that $\tau_j \cap \mathcal{N} = \tau_j \cap A_j$ is a single point. Therefore $\varphi_\N([\tau_j]) = \varphi_{\mathcal{A}}([\tau_j])$ for each $j$ and $\varphi_\mathcal{A} \colon \pi_1(W) \to F_r$ is surjective. We will show that $\ker \varphi_\N \subset \ker \varphi_\mathcal{A}$.
	
	Let $[\alpha] \in \ker \varphi_\N$ and write $\alpha = \alpha_{1} \cdot \ldots \cdot \alpha_{k}$ as in Lemma \ref{lemma:cutting_loops_into_arcs} with respect to the system~$\N$.
	We proceed by induction on $k$, which is even since $\varphi_\N([\alpha]) =1$. If $k=0$, then  $\alpha \cap \N = \varnothing$, so $\alpha \in W | \N \subset W | \mathcal{A}$
	and therefore $[\alpha] \in \ker \varphi_\mathcal{A}$. Suppose that any element in $\ker\varphi_\N$ represented by a loop which can be written as the concatenation of less than $k$ paths as in
	Lemma \ref{lemma:cutting_loops_into_arcs} is also contained in $\ker\varphi_\mathcal{A}$. Let $\alpha = \alpha_{1} \cdot \ldots \cdot \alpha_{k}$ for $[\alpha] \in \ker\varphi_\N$. Since
	$1 = \varphi_\N([\alpha]) = a^{\epsilon_1}_{i_1}\ldots a^{\epsilon_k}_{i_k}$, there is an index $m$ such that $a_{i_m} = a_{i_{m+1}}$ and $\epsilon_{m+1} = - \epsilon_m$, so $i_m=i_{m+1} =: j$.
	Thus both the paths $\alpha_m$ and $\alpha_{m+1}$ intersect the same component $X$ of $N_j$ since there are no paths as in Lemma \ref{lemma:connected_sum_of_submanifolds}.
	Obviously, we may extend slightly the tubular neighbourhood of $X$ and assume that the beginning of the path $\alpha_m$ and the end of $\alpha_{m+1}$ are in $P_t(X)$ for some $t\notin[-1,1]$. Since $X$ is connected,
	so also is $P_t(X)$, there is an arc $\gamma$ in $P_t(X)$ joining these two points. Thus we may define the loop
	$$
	\beta = \alpha_1 \cdot \ldots \cdot \alpha_{m-1}  \cdot (\gamma \cdot \alpha_{m+2})  \cdot \alpha_{m+3} \cdot \ldots \cdot \alpha_k
	$$
	which has $k-2$ paths as in Lemma \ref{lemma:cutting_loops_into_arcs}.  Write $\varphi_\N([\alpha]) = \omega \cdot a_j^{\epsilon_m}a_j^{-\epsilon_m} \cdot \omega'$.
	Evidently, $\varphi_\N([\beta]) = \omega\omega' = 1$ and by induction hypothesis $\varphi_\mathcal{A}([\beta])=1$. It is clear that in both the cases $X=A_j$ and $X\neq A_j$	we get $\varphi_\mathcal{A}([\alpha]) = \varphi_\mathcal{A}([\beta]) =1$, so $[\alpha] \in \ker\varphi_\mathcal{A}$. By induction $\ker\varphi_\N \subset \ker\varphi_\mathcal{A}$.
	
	Therefore $\varphi_\mathcal{A} = \eta \circ \varphi_\N$ for some epimorphism $\eta\colon F_r \to F_r$. Since free groups are Hopfian (see \cite{Bogopolski}), $\eta$ is an isomorphism, so $\ker \varphi_\N = \ker \varphi_\mathcal{A}$. Because $[\tau_j]$'s  generate a subgroup of $\pi_1(W)$ mapped isomorphically onto $F_r$ by $\varphi_\N$ and $\varphi_\mathcal{A}$ on which they are equal,
	we obtain $\varphi_\mathcal{A} = \varphi_\N$ everywhere and the theorem is proved.
\end{proof}

\section{Systems of hypersurfaces in the classification of epimorphisms} \label{section:cobordism_of_systems_and_classification_of_epimorphism}

Hereafter, $\Sigma_g$ and $S_g$ denote respectively an orientable and non-orientable closed surface of genus~$g$.

Let $G$ be a finitely generated group and $\varphi \colon G \to F_r$ be an epimorphism. The number $r$ is called the \emph{rank} of an epimorphism~$\varphi$. The \emph{corank} of $G$ is defined as the largest rank of an epimorphism from $G$ onto a free group and it is denoted by $\corank(G)$. Since $G$ is finitely generated it is well-defined and 
$$
\corank(G) \leq \rank_\Z \ab(G),
$$
where $\ab(G)$ is the abelianization of $G$. For more information about the corank and its properties we refer to \cite{Cornea, Gelbukh:corank, Gelbukh:filomat, Jaco,Michalak-DCG}. In the case when $G=\pi_1(X)$ the corank of $G$ is also called the first non-commutative Betti number of $X$ (cf. \cite{Gelbukh:corank}). We only recall that $\corank(\pi_1(\Sigma_g))=g$ and $\corank(\pi_1(S_g))= \left\lfloor g/2 \right\rfloor$, the floor of $g/2$.

Grigorchuk, Kurchanov and Zieschang in \cite{Grigorchuk-Kurchanov, Grigorchuk-Zieschang} studied epimorphisms onto free groups from fundamental groups of compact surfaces. As in their papers,
we call two homomorphisms $\varphi, \psi  \colon G \to H$ \emph{equivalent}, and we denote it by  $\varphi \sim \psi$, if there exist isomorphisms $\nu \colon G \to G$ and
$\eta \colon H \to H$ such that $\varphi \circ \nu = \eta \circ \psi$. They are called \emph{strongly equivalent} if one can choose $\eta=\operatorname{id}_{H}$.  In this case we write $\varphi \simeq \psi$. Obviously, it implies that $\varphi \sim \psi$. We are interested in the case $H=F_r$.

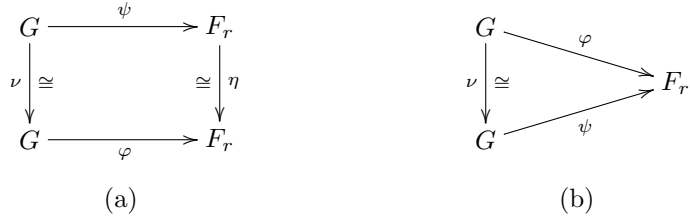
\begin{figure}[h]
	$$
	\begin{xy}
		(0,15)*+{G}="G2";
		(0,0)*+{G}="G1";
		(25,15)*+{F_r}="F1";
		(25,0)*+{F_r}="F2";
		{\ar@{->}_{\varphi} "G1"; "F2"};%
		{\ar@{->}_{\nu}^{\cong} "G2"; "G1"};%
		{\ar@{->}^{\psi} "G2"; "F1"};%
		{\ar@{->}^{\eta}_{\cong}  "F1"; "F2"};
		(12,-8)*+{\text{(a)}}="a";
		(60,15)*+{G}="Ga";
		(60,0)*+{G}="Gb";
		(85,7.5)*+{F_r}="F";
		{\ar@{->}^{\psi} "Ga"; "F"};%
		{\ar@{->}_{\nu}^{\cong} "Ga"; "Gb"};%
		{\ar@{->}_{\varphi} "Gb"; "F"};%
		(72,-8)*+{\text{(b)}}="b";
	\end{xy}
	$$
	\caption{Equivalence (a) and strong equivalence (b) of epimorphisms onto free groups.}\label{figure:factorization_and_equivalence}
\end{figure}

In this section we apply the results of the previous section to the problem of classification of epimorphisms onto free groups up to equivalence and strong equivalence. In particular, we give an alternative proof of the following theorem.

\begin{theorem}[Grigorchuk--Kurchanov--Zieschang {\cite{Grigorchuk-Kurchanov, Grigorchuk-algebra,Grigorchuk-Zieschang}}]\label{thm:grigorchuk}
	If $\Sigma$ is a closed surface of Euler characteristic $\chi(\Sigma)=2-k$ and $1 \leq r \leq  \left\lfloor \frac{k}{2} \right\rfloor =\corank(\pi_1(\Sigma))$, then there exist finite numbers $p$ and $q$
	of classes of epimorphisms $\pi_1(\Sigma) \to F_r$ with respect to equivalence and strong equivalence, respectively. More precisely,
	\begin{enumerate}[(1)]
		\item if $\Sigma$ is orientable, then $p=q=1$,
		\item if $\Sigma = S_k$ is non-orientable, then we have:
		\begin{enumerate}[(a)]
			\item $p=q=1$ if the genus $k=2m+1$ is odd,
			\item $p=2$ and $q=2^r$ if the genus $k=2m$ is even and $r < m$,
			\item $p=1$ and $q=2^r-1$ if the genus $k=2m$ is even and $r=m$.
		\end{enumerate}
	\end{enumerate}
\end{theorem}

\begin{proposition}[{\cite{Grigorchuk-algebra}}]\label{thm:grigorchuk_free}
	For $m \geq r$ there exists only one class of epimorphisms $F_m \to F_r$ up to strong equivalence.
\end{proposition}

It should be noted that the Poincar{\'e} conjecture is equivalent to the classification of some pairs of epimorphisms onto free groups, which shows the importance of their studies.

\begin{theorem}[Stallings--Jaco--Waldhausen--Hempel, \cite{Hempel, Jaco-splittings}] 
	\ \\
	The Poincar{\'e} conjecture holds if and only if for each $g\geq 2$ any two epimorphisms $\pi_1(\Sigma_g) \to F_g \times F_g$ are equivalent.
\end{theorem}

\subsection{Systems of hypersurfaces up to framed cobordism and diffeomorphism}\label{section:calculation_of_cobordisms}

Let us denote by $\mathcal{H}_r(M)$ the set of all independent and regular systems of hypersurfaces in $M$ of size $r$ which omit the basepoint, and by $\H_r^{fr}(M)$ the set of framed cobordism classes of elements of $\H_r(M)$.
On each of these sets there is a natural action of $\diffd(M)$, the set of self-diffeomorphisms of $M$ which preserve the basepoint, so we may form the orbit space $\H^{fr}_r(M)/_{\diffd(M)}$.  Note that if $h \in \diffd(M)$, then a system $\N=(N_1,\ldots,N_r)$ and its image $h(\N) = (h(N_1),\ldots,h(N_r))$ induce strongly equivalent homomorphisms.

Moreover, for groups $G$ and $H$ denote by $\epi(G,H)$ the set of all epimorphisms $G \to H$. 

We have the natural map $\Theta \colon \H_r(M) \to \epi(\pi_1(M),F_r)$ which sends a system $\N$ into the induced epimorphism~$\varphi_\N$. 
By Proposition \ref{proposition:framed_cobordant_systems=the_same_induced_epimorphisms} it factorizes through the injective map $\overline{\Theta} \colon \H_r^{fr}(M) \to \epi(\pi_1(M),F_r)$. Theorem \ref{theorem:epimorphism_is_induced_by_regular_and_independent_system} states that both these mappings are also surjective.

\begin{corollary}
	The map $\overline{\Theta} \colon \H^{fr}_r(M) \to \epi(\pi_1(M),F_r)$ is a bijection between the set of all framed cobordism classes of regular and independent systems of hypersurfaces of size $r$ in $M$ and the set of all epimorphisms from $\pi_1(M)$ onto the free group of rank $r$. \qed
\end{corollary}

Now, let us consider the strong equivalence relation $\simeq$ on $\epi(G,F_r)$. The composition 
$$
\H^{fr}_r(M)\to \epi(\pi_1(M),F_r) \to \epi(\pi_1(M),F_r)/_\simeq
$$
is still surjective and it factorizes through the map $\overline{\overline{\Theta}} \colon \H^{fr}_r(M)/_{\diffd(M)} \to \epi(\pi_1(M),F_r)/_\simeq$.

\begin{corollary}
	The number of strong equivalence classes of epimorphisms $\pi_1(M)\to F_r$ is not greater than the cardinality of  $\H_r^{fr}(M)/_{\diffd(M)}$.
\end{corollary}

The question is when the latter set is finite. It is for example the case for the surface groups.

\begin{proposition}\label{proposition:bijection_for_surfaces_epi_and_systems}
	For a closed surface $\Sigma$ the map $\overline{\overline{\Theta}} \colon \H^{fr}(\Sigma)/_{\diffd(\Sigma)} \to \epi(\pi_1(\Sigma),F_r)/_\simeq$ is a bijection.
\end{proposition}

\begin{proof}
	We know that it is surjective. For injectivity it suffices to note that by Dehn--Nielsen Theorem (see \cite{Collins-Zieschang}) any automorphism of $\pi_1(\Sigma)$ can be represented by a self-diffeomorphism of~$\Sigma$. If $\varphi_\N$ and $\varphi_{\N'}$ are strongly equivalent by $\eta = h_\#$ induced by $h\in \diffd(\Sigma)$, then $\varphi_\N = \varphi_{\N'} \circ \eta = (f_{\N'}\circ h)_\# = (f_{h^{-1}(\N')})_\# = \varphi_{h^{-1}(\N')}$, so $\N$ and $h^{-1}(\N')$ are framed cobordant.
\end{proof}

\begin{remark}
	The same fact is true for any manifold $M$ for which any automorphism of $\pi_1(M)$ is induced by some element of $\diffd(M)$. By Mostow Rigidity Theorem it is the case for hyperbolic manifolds of dimension at least~$3$.
\end{remark}

Now, our aim is to calculate $\H^{fr}_r(\Sigma)/_{\diffd(\Sigma)}$. We need the following series of three lemmas.

\begin{lemma}\label{lemma:self_map_of_non_orientable_surface_which_reverse_orientation_on_boundary}
	Let $\Sigma$ be a non-orientable compact surface with $\partial\Sigma \neq \varnothing$ and $S\subset \partial\Sigma$ be a connected component. Then there exist $h\in\diffd(\Sigma)$ such that $h(S)=S$,  $h|_S$ is orientation-reversing and $h|_{\partial\Sigma\setminus S} = \id_{\partial\Sigma\setminus S}$.
\end{lemma}

\begin{proof}
	First, assume that $\Sigma$ is the projective plane $\rp$ with one disc $B$ removed, i.e. $\Sigma = \rp \setminus \int B$ and  $S=\partial B$. Let $D\subset \int \Sigma$ be another disc and $\Sigma'=\Sigma\setminus D \cup B$ be a M\"obius band. Fix a parametrization $\Sigma' \cong [-1,1]\times[0,1] / (t,0) \sim (-t,1)$ for $t\in [-1,1]$ such that $S \subset \int \Sigma'$ is symmetric with respect to the core $\{0\}\times[0,1]$, i.e. if $(t,x)\in S$, then $(-t,x)$ is also in $S$. Then $h' \colon \Sigma' \to \Sigma'$ defined by $h'(t,x) = (-t,x)$ is a self-diffeomorphism such that $h'|_S \colon S \to S$ has degree $-1$, so it is orientation-reversing, but on $\partial \Sigma'$ it is orientation-preserving, so isotopic to the identity. Thus we can extend $h'$ to $h\colon \rp \to \rp$ such that $h(B)=B$ and $h|_S$ has degree $-1$, and take $h|_\Sigma$.
	
	In general case, glue a disc $B$ and $\Sigma$ along $S$ and take a diffeomorphism $\Sigma \cup_S B \to \Sigma'' \# \rp$ such that $B \subset \Sigma' \subset  \rp$ as before. The lemma follows from the first case.
\end{proof}

\begin{lemma}\label{lemma:when_non_independent_system_induces_no_surjection}
	Let $\N = (N_1 \cup N_2)$ be a system of size $1$ in a manifold $M$ such that $N_1$ and $N_2$ are connected. If $M\setminus N_1$ and $M\setminus N_2$ are connected, but $M|\N$ is disconnected, then $\varphi_\N \colon \pi_1(M)\to \Z$ is not surjective.
\end{lemma}

\begin{proof}
	Assume that $\varphi_\N$ is an epimorphism. Then there is a loop $\alpha$ in a general position to $\N$ such that $\varphi_\N ([\alpha]) = \pm 1$. As in Lemma \ref{lemma:cutting_loops_into_arcs} write $\alpha = \alpha_1 \ldots \alpha_k$ as a concatenation of paths $\alpha_i$, each of which intersects $\N$ in a single point. Therefore 
	$$
	1 \equiv \pm 1 = \varphi_\N ([\alpha]) \equiv k \mod 2,
	$$
	so $k$ is odd. By the assumption $M|\N$ has exactly two components. Since $N_1$, $N_2$, $M\setminus N_1$ and $M\setminus N_2$ are connected, each $\alpha_i$ joins both the components of $M|\N$. Thus $k$ is even, because $\alpha$ is a loop, so it starts and ends at the same point. It gives a contradiction, so $\varphi_\N$ is not surjective. 
\end{proof}

\begin{remark}
	While we know that independent systems induce surjective homomorphisms, non-independent systems can induce both surjective or not surjective homomorphisms. The above lemma shows when $\varphi_\N$ is not an epimorphisms and it can be generalized for other similar situations.
\end{remark}

\begin{lemma}\label{lemma:changing_system_to_regular_and_the_same_orientantion_of_the_complement}
	For an independent system $\N= (N_1\cup\ldots \cup N_r)$  of size $1$ in a manifold $M$ there exists a regular and independent system $\N' = (N')$ which is framed cobordant to $\N$, so $\varphi_\N = \varphi_{\N'}$. Moreover,
	\begin{enumerate}[(1)]
		\item The complement $M|\N'$ can be non-orientable if $M|\N$ is non-orientable.
		\item The complement $M|\N'$ is orientable if $M|\N$ is orientable and $M|\N \cup P(N_i)$ is non-orientable for each~$i$. 
	\end{enumerate}
\end{lemma}

\begin{proof}
	The construction of $\N'$ is performed as in the proof of Lemma \ref{lemma:connected_sum_of_submanifolds} by using arcs $\gamma$ connecting components of $\N$. They can be found since $\N$ is independent.
	
	Consider a two-sheeted orientation cover $\pi \colon \widetilde{M}\to M$, where 
	$$
	\widetilde{M} := \{\mu_x \,|\, x \in M \text{ and } \mu_x \in H_n(M,M\setminus\{x\}) \text{ is a local orientation at $x$} \}.
	$$
	
	For 1), if $M|\N$ is non-orientable, then there is a loop $\alpha$ in $M|\N$ which reverses the orientation, which means that it lifts to a path in $\widetilde{M}$ which joins two different local orientations at the basepoint. Since $M|\N \setminus \im \alpha$ is connected, we may perform the construction of $\N'$ in this space. Then $\alpha$ is also contained in $M|\N'$, so it is non-orientable.
	
	Now, assume that $M|\N$ is orientable, but $M|\N \cup P(N_i)$ is non-orientable for each~$i$. To obtain a contradiction, suppose that $M|\N'$ is non-orientable, so there is a loop $\alpha$ in $M|\N'$ which reverses the orientation and we may assume that it is in general position to $\N$. Using Lemma \ref{lemma:cutting_loops_into_arcs} write $\alpha = \alpha_1 \ldots \alpha_k$ as a concatenation of paths $\alpha_i$, each of which intersects $\N$ in a single point. Note that since $\alpha$ is in $M|\N'$, it intersects $\N$ only when it goes into or leaves a tubular neighbourhood $P(\gamma)$ of some arc $\gamma$ as mentioned in the beginning of the proof. Therefore, as in the proof of Lemma \ref{lemma:connected_sum_of_submanifolds}, if $\alpha$ intersects $\N$ going inside $P(\gamma)$, then it needs to leave $P(\gamma)$ again intersecting $\N$. Thus $k$ is even.
	
	For any $i$ consider $\alpha_i$ which intersects $N_i$ in a point $x$ and take a small closed disc $D$ in $M$ around $x$ such that the cover $\pi$ is trivial over $D$ and $\partial D \cap \im \alpha_i = \{x_1,x_2\}$, where $x_1$ and $x_2$ lie on the different sides of $N_i$ such that $\alpha_i$ goes from $x_1$ to $x_2$. By the assumption, there is a reversing-orientation loop $\beta_i$ in $M|\N \cup P(N_i)$ intersecting $\N$ only once at $x$ and we may assume that its image agrees with the image of $\alpha_i$ on $D$. We take a loop $\alpha'$ which differs from $\alpha$ only on the segment of $\alpha_i$ between $x_1$ and $x_2$, where it goes as $\beta_i$ outside $D$. Note that the local orientations in $x_2$ assigned by lifts of $\alpha$ and $\alpha'$ are opposed. Repeating this for each $\alpha_i$ we obtain a loop $\alpha''$ which omits $\N$ and which is still orientation-reversing since we changed the local orientations by $\beta_i$ an even number of times. This contradicts the fact that $M|\N$ is orientable and proves 2).
\end{proof}

\begin{remark}
	In fact, in 1) the complement $M|\N'$ is always non-orientable if $M|\N$ is non-orientable. For this, if $\N''$ is any other regular and independent system framed cobordant to $\N$ such that $M|\N''$ is orientable, then it is also framed cobordant to $\N'$, but it is a contradiction by the next proposition. 
\end{remark}

\begin{proposition}\label{proposition:not_cobordant_if_one_complement_is_orientable_and_second_not}
	Let $M$ be a non-orientable manifold and let $\N$ and $\N'$ be two regular and independent systems of hypersurfaces in $M$ of the same size $r$ such that $M|\N$ is orientable, but $M|\N'$ is non-orientable. Then $\N$ and $\N'$ are not framed cobordant.
\end{proposition}

\begin{proof}
	Let $\N=(N_1,\ldots,N_r)$ and $\N' = (N'_1,\ldots,N'_r)$. It is clear that we may assume that $\N$ satisfies the conditions in Lemma \ref{lemma:changing_system_to_regular_and_the_same_orientantion_of_the_complement} 2) since a framed cobordism between $\N$ and $\N'$ implies a framed cobordism between the systems $\N_* = (N_{i_1},\ldots,N_{i_k})$ and $\N'_*=(N'_{i_1},\ldots,N'_{i_k})$, where  $i_1 < \ldots <i_k$ are all indices such that $M|\N \cup P(N_{i_j})$ is non-orientable. We will show that $\N_*$ and $\N'_*$ are not framed cobordant even as submanifolds, not as systems of hypersurfaces. For this we  use Lemma \ref{lemma:changing_system_to_regular_and_the_same_orientantion_of_the_complement} for $(N_{i_1} \cup \ldots \cup N_{i_k})$ and $(N'_{i_1} \cup \ldots \cup N'_{i_k})$ to assume that $r=1$.
	
	So now, each of $\N$ and $\N'$ is just a non-separating connected $2$-sided submanifold in $M$, $M|\N$ is orientable and $M|\N'$ is non-orientable. Suppose that $W \subset M\times[0,1]$ is a framed cobordism between $\N$ and $\N'$. Take the orientation cover $\pi\colon\widetilde{M}\to M$ and take the lifts $\widetilde{\N} := \pi^{-1}(\N)$ and  $\widetilde{\N'} := \pi^{-1}(\N')$. Moreover, by the property of $\pi$ the complement $\widetilde{M} | \widetilde{\N}$ has two connected components since $M|\N$ is orientable, and $\widetilde{M} | \widetilde{\N'}$ is connected since $M|\N'$ is non-orientable. The cobordism $W$ is lifted to the framed cobordism $\widetilde{W} := (\pi \times \id_{[0,1]})^{-1}(W) \subset \widetilde{M} \times [0,1]$ between $\widetilde{\N}$ and $\widetilde{\N'}$. Therefore $\varphi_{\widetilde{\N}} = \varphi_{\widetilde{\N'}} \colon \pi_1(\widetilde{M}) \to \Z$ and $\varphi_{\widetilde{\N'}}$ is surjective, because $\widetilde{\N'}$ is independent. However, $\varphi_{\widetilde{\N}}$ is not surjective, which gives a contradiction.
	
	To see this, note that $\widetilde{\N}$ can have one or two components. If $\widetilde{\N}$ is connected, then $\varphi_{\widetilde{\N}}$ is evidently not surjective, since $\widetilde{M} | \widetilde{\N}$ is not connected. In the second case when $\widetilde{\N}$ has two components, we use Lemma \ref{lemma:when_non_independent_system_induces_no_surjection} together with the fact that $\widetilde{M} | \widetilde{\N} \to M|\N$,  the restriction of $\pi$, is also the orientation cover of $M|\N$.
	
	Thus $\N$ and $\N'$ are not framed cobordant.
\end{proof}

\begin{remark}
	The above proposition is easily not true for not regular systems of hypersurfaces. To construct an example, it suffices to take a system $\N = (N_1 \cup N_2)$ of size $1$ consisting of two non-separating framed circles in $M = \Sigma_1 \# S_2$, the connected sum of the torus and Klein bottle, such that $N_1 \subset \Sigma_1$ and $N_2 \subset S_2$ are disjoint from the discs used in the connected sum operation. Since $N_2$ has a trivial normal bundle and $S_2 | N_2$ is orientable, there is an orientation-reversing loop $\beta$ in $S_2$ intersecting $N_2$ in a single point. Obviously, there is also an orientation-preserving loop $\alpha$ in $\Sigma_1$ intersecting $N_1$ in a single point. Moreover, one can take an arc $\gamma$ in $M$ joining $\alpha \cap N_1$ and $\beta \cap N_2$ and disjoint from $\alpha$ and $\beta$ outside these points. Performing the connected sum of $N_1$ and $N_2$ along $\gamma$ we obtain $\N' = (N')$ such that $\varphi_\N = \varphi_{\N'}$ by Lemma \ref{lemma:connected_sum_of_submanifolds}, so $\N$ and $\N'$ are framed cobordant. However, $M|\N$ is orientable, but $M|\N'$ is non-orientable because the concatenation $\alpha \cdot \gamma \cdot  \beta$ is an orientation-reversing loop which can be homotoped to lie in $M|\N'$.
\end{remark}

Now, we can provide an alternative proof of Grigorchuk--Kurchanov--Zieschang Theorem (Theorem \ref{thm:grigorchuk}) for the strong equivalence relation. First, let us make a short preparation.

Let $\N=(N_1,\ldots,N_r)$ and $\N'=(N'_1,\ldots,N'_r)$ be two arbitrary regular and independent systems of hypersurfaces in a closed surface $\Sigma$. Thus all $N_i$, $N'_i$ are circles. Assume that $\Sigma|\N$ and $\Sigma|\N'$ are diffeomorphic.  By homogeneity of manifolds, take a diffeomorphism $h' \colon \Sigma|\N \to \Sigma|\N'$ which sends $P_{\pm 1}(N_i)$ onto $P_{\pm 1}(N'_i)$. Glue all tubes $P(N_i) \cong [-1,1]\times N_i$ to $\Sigma|\N$ along $\{-1\} \times N_i$ to obtain a surface $\overline{\Sigma}$ with $2r$ boundary components $\{1\}\times N_i$ and $P_1(N_i)$, $i=1,\ldots,r$. Let 
$$
\xi_i \colon  P_{1}(N_i) \to \{1\} \times N_i 
$$
be a gluing map which leads to $\Sigma$. Analogously, we define $\overline{\Sigma'}$ and take $\xi'_i$ for $\N'$. Extend $h'$ to a diffeomorphism $\bar{h} \colon \overline{\Sigma} \to \overline{\Sigma'}$ using $P(N_i) \cong [-1,1] \times \es^1 \cong P(N'_i)$, so $\bar{h}(\N)=\N'$. It follows easily that $\bar{h}$ induces $h \in \diffd(\Sigma)$ after performing gluing operations via $\xi_i$ and $\xi'_i$  if and only if $\bar{h}^{-1}\circ \xi'_i \circ \bar{h}|_{P_1(N_i)}$ is isotopic to $\xi_i$ for each $i=1,\ldots,r$. If it is the case, then  $h(\N) = \N'$, so $\N$ and $\N'$ are the same elements in $\H^{fr}_r(\Sigma)/_{\diffd(\Sigma)}$.

\begin{theorem}\label{theorem:calculations_of_cobordisms_for_surfaces}
	Let $\Sigma$ be a closed surface, let $r$ be an integer such that $1 \leq r \leq \corank(\pi_1(\Sigma))$ and set $q = \left|\H^{fr}_r(\Sigma)/_{\diffd(\Sigma)}\right|$.
	\begin{enumerate}[(1)]
		\item If $\Sigma$ is orientable or non-orientable of odd genus, then $q=1$. 
		\item If $\Sigma=S_{2m}$ is non-orientable of genus $2m$, then
		\begin{itemize}
			\item if $r<m$, then $q= 2^r$,
			\item if $r=m$, then $q= 2^r-1$.
		\end{itemize}
	\end{enumerate}
	As a consequence, $q$ is the number of strong equivalence classes of epimorphisms $\pi_1(\Sigma)\to F_r$ as in Theorem~\ref{thm:grigorchuk}.
\end{theorem}

\begin{proof}

	We use the above notation. If $\Sigma$ is orientable, then $\Sigma|\N$ and $\Sigma|\N'$ are diffeomorphic surfaces and we may assume that the diffeomorphism $h'$ is orientation-preserving. Since $\Sigma$ is orientable, all maps $\xi_i$ and $\xi'_i$ are also orientation-preserving, so we obtain $h\in\diffd(\Sigma)$ such that $h(N_i) = N'_i$. Therefore $q=1$.

	
	Now assume that $\Sigma$ is non-orientable of odd genus. Then $\Sigma|\N$ and $\Sigma|\N'$ are compact surfaces with $2r$ boundary components and of the same odd Euler characteristic, so they are also non-orientable. Using Lemma \ref{lemma:self_map_of_non_orientable_surface_which_reverse_orientation_on_boundary} we may change $h'$, by the composition with another diffeomorphism, so that $\bar{h}^{-1}\circ \xi'_i \circ \bar{h}|_{P_1(N_i)}$ and $\xi_i$ are isotopic. As before, it implies that $q=1$.

	
	Finally, let $\Sigma = S_{2m}$ be non-orientable of even genus $2m$.
	For any non-empty subset $I \subset \{1,\ldots,r\}$ it is easy to construct a system $\N_I$ such that $\Sigma|\N_I$ is orientable and gluing maps $\xi_i^I$ (defined as before) are orientation-reversing only for $i\in I$. We omit the case when $I=\varnothing$ since then $\Sigma$ would be orientable. Moreover, for $r<m$ we denote by $\N_0$ a system for which $\Sigma|\N_0$ is non-orientable. Note that if $r=m$, then $\Sigma|\N$ is always the sphere with $2r$ open discs removed, so it is orientable. 
	
	By the previous considerations it is clear that the systems $\N_I$ for $\varnothing \neq I \subset \{1,\ldots,r\}$ and $\N_0$ for $r<m$ represent all elements of $\H^{fr}_r(\Sigma)/_{\diffd(\Sigma)}$ (for the case when $\Sigma|\N$ is non-orientable we use Lemma \ref{lemma:self_map_of_non_orientable_surface_which_reverse_orientation_on_boundary} as before). Thus $q \leq 2^r$  for $r<m$ and $q \leq 2^r-1$ for $r=m$. We will show that they are different elements of $\H^{fr}_r(\Sigma)/_{\diffd(\Sigma)}$. It will be done if we show that the systems are not framed cobordant to each other. 
	
	By Proposition \ref{proposition:not_cobordant_if_one_complement_is_orientable_and_second_not} we know that $\N_0$ is not framed cobordant to any $\N_I$. If we have two systems $\N_I=(N^I_1,\ldots,N^I_r)$ and $\N_J=(N^J_1,\ldots,N^J_r)$ for $I\neq J$, then we may assume that there is an index $1 \leq j \leq r$ such that $j \notin I$, but $j\in J$, so $\xi_j^I$ is orientation-preserving, but $\xi_j^J$ is orientation-reversing. If $I=\{i_1,\ldots,i_k\}$, form the systems $\N_I^* = (N^I_{i_1}, \ldots, N^I_{i_k})$ and $\N_J^* = (N^J_{i_1}, \ldots, N^J_{i_k})$. By the construction, $\Sigma|\N_I^*$ is orientable, but $\Sigma|\N_J^*$ is non-orientable. Again by Proposition \ref{proposition:not_cobordant_if_one_complement_is_orientable_and_second_not} we get that $\N_I^*$ and $\N_J^*$ are not framed cobordant, so also $\N_I$ and $\N_J$ cannot be framed cobordant and the proof is complete.
	
	The last statement follows by Proposition \ref{proposition:bijection_for_surfaces_epi_and_systems}.
\end{proof}

\begin{corollary}
	With the above notation, 		
	\begin{equation*}
		\mathcal{H}^{fr}_r(S_{2m})/_{\diffd(S_{2m})} =  \begin{cases}
			\left\{[\N_0],[\N_I]\ \colon\ \varnothing \neq I \subset \{1,\ldots,r\}\right\}	        &\text{\ \ for $r<m$,}\\
			\{[\N_I]\ \colon\ \varnothing \neq I \subset \{1,\ldots,r\}\}	        &\text{\ \ for $r=m$.}
		\end{cases}
	\end{equation*}\qed
\end{corollary}

\subsection{Analogue of Nielsen transformations for systems of hypersurfaces}

We have found out that strong equivalence classes of epimorphisms $\pi_1(M) \to F_r$ can be described by elements of $\H^{fr}_r(M)/_{\diffd(M)}$. In this section we show how to get equivalence classes from them.

It is known that the automorphism group $\operatorname{Aut}(F_r)$ of a finitely generated free group $F_r$ is generated by \emph{elementary Nielsen transformations} (see e.g.~\cite{Bogopolski}). On a given ordered basis $(a_1,\ldots,a_r)$ we define them as follows:
\begin{enumerate}[(T1)]
	\item $n_\sigma \colon (a_1,\ldots,a_r) \mapsto (a_{\sigma(1)},\ldots,a_{\sigma(r)})$ for some permutation $\sigma \in S_r$;
	\item $n_i \colon (a_1,\ldots,a_r) \mapsto (a_1,\ldots,a_{i-1},a^{-1}_i, a_{i+1},\ldots,a_r)$ for $i \in \{1,\ldots,r\}$;
	\item $n_{ij}\colon (a_1,\ldots,a_r) \mapsto (a_1,\ldots, a_{i-1},a_ia_j, a_{i+1}, \ldots,a_r)$ which replaces $a_i$ by $a_ia_j$ for some $i\neq j$.
\end{enumerate}

Note that the transformation (T1) can be obtained from the other two transformations, but it is convenient to use. Thus we have three types of automorphisms: $n_\sigma, n_i, n_{ij} \in \aut(F_r)$.

\begin{definition}\label{nielsen_transformations_for_systems}
	Let $\N=(N_1,\ldots,N_r)$ be an independent and regular system of hypersurfaces in a closed manifold $M$. We define analogous operations on $\H_r(M)$:
	
	\begin{enumerate}[(H1)]
		\item $\N \mapsto \N^\sigma := (N_{\sigma(1)},\ldots,N_{\sigma(r)})$ for some permutation $\sigma \in S_r$;
		\item $\N \mapsto \N^i$ is obtained by changing the framing of the submanifold $N_i$ to the one with opposite orientation;
		\item $\N \mapsto \N^{ij}$ is obtained for $i\neq j$ by replacing $N_j$ by $N_j \#_\gamma P_1(N_i)$, where $\gamma$ is an arc as in Lemma \ref{lemma:connected_sum_of_submanifolds} which intersects $\N$ only in two points and joins $N_j$ and $P_1(N_i)$ from the same side.
	\end{enumerate}
\end{definition}

An arc $\gamma$ in (H3) always exists since $\N$ is independent. Then for the obtained system $\N^{ij}$ we take smaller tubular neighbourhoods to be disjoint, e.g. $P_{[-1,\frac{1}{2}]}(N_i) \cong [-1,\frac{1}{2}]\times N_i$. By Lemma \ref{lemma:connected_sum_of_submanifolds} the homomorphism $\varphi_{\N^{ij}}$ is the same as induced by the system $(N_1,\ldots,N_i,\ldots,N_j\cup P_1(N_i),\ldots,N_r)$, so it is clear by the definition that $\varphi_{\N^{ij}} = n_{ij} \circ \varphi_\N$. Therefore $\varphi_{\N^{ij}}$ is surjective and since obviously $\N^{ij}$ is regular, by Proposition \ref{proposition:epi_and_regular_is_independent} it is also independent, so the operation (H3) on $\H_r(M)$ is well defined. It does not depend on the choice of $\gamma$ up to framed cobordism.

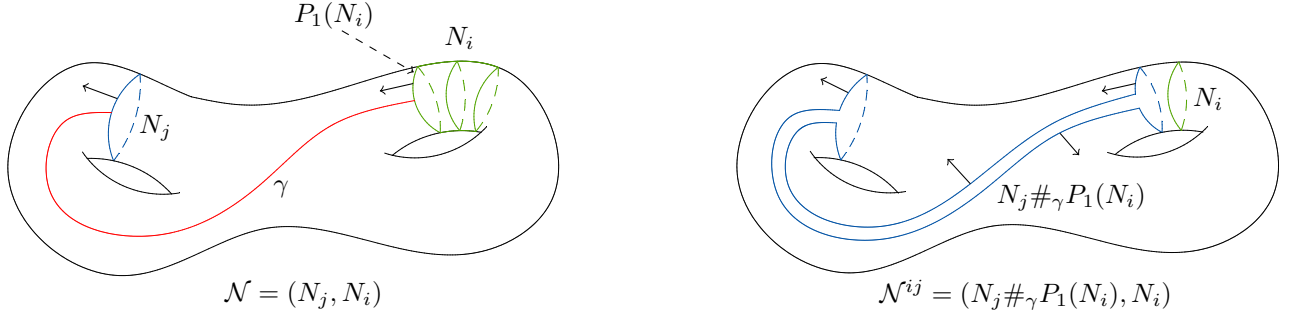
\begin{figure}[h]
	\centering

	\begin{tikzpicture}[scale=1]
		\node at (0,0) {\includegraphics[width=215pt]{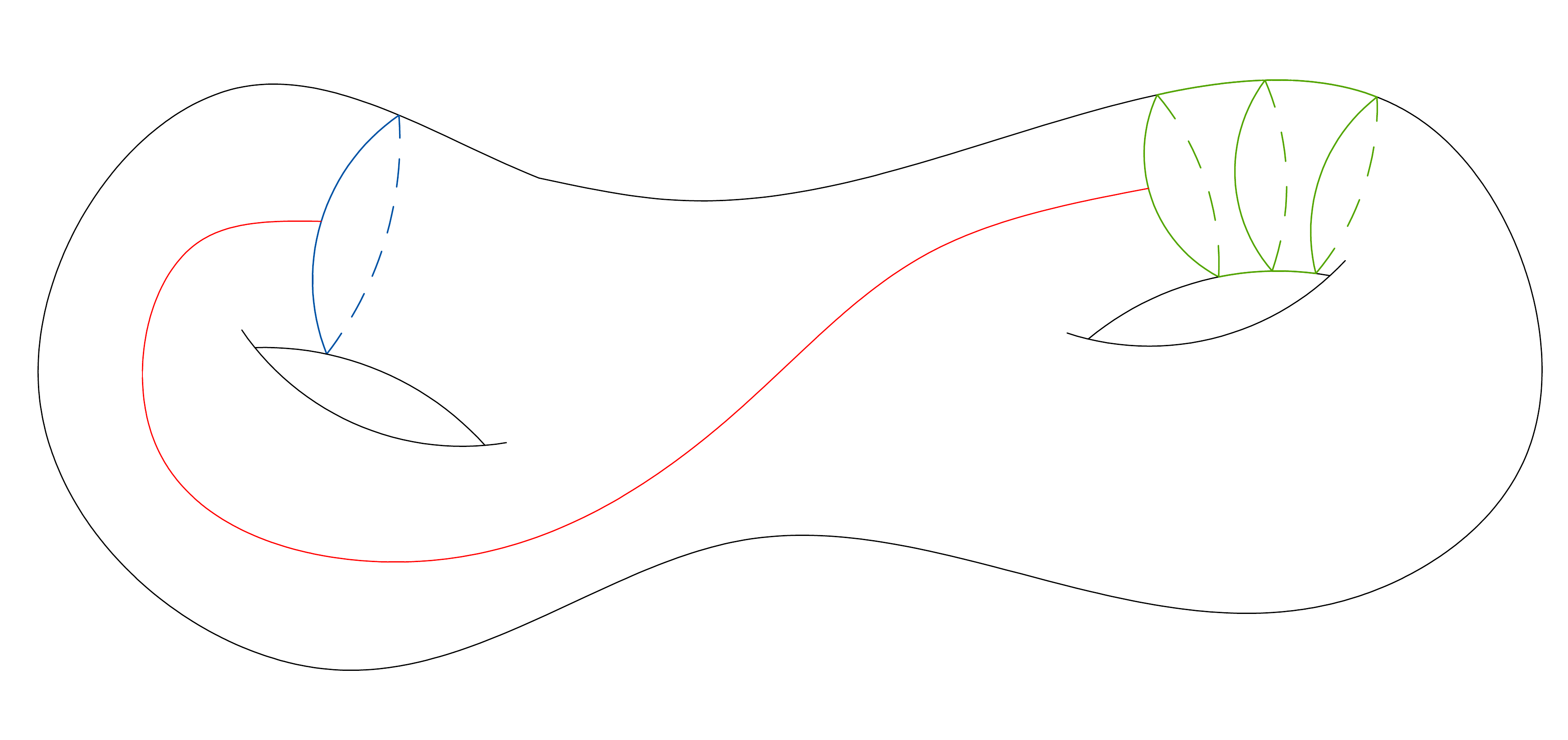}};
		
		\draw (-1.65,0.52) node {$N_j$};
		\draw (2.35,1.7) node {$N_i$}; 
		\draw (0.7,2) node {$P_1(N_i)$}; 
		\draw (0,-0.3) node {$\gamma$}; 
		
		\draw [->] (-2.155,0.9) to (-2.61,1.07); 
		\draw [->] (1.74,1.1) to (1.3,1); 
		\draw [dashed,->] (0.8,1.8) to (1.745,1.25);

		\draw (0.3,-1.8) node {$\N=(N_j,N_i)$}; 
		
		

		\node at (8.1,0) {\includegraphics[width=215pt]{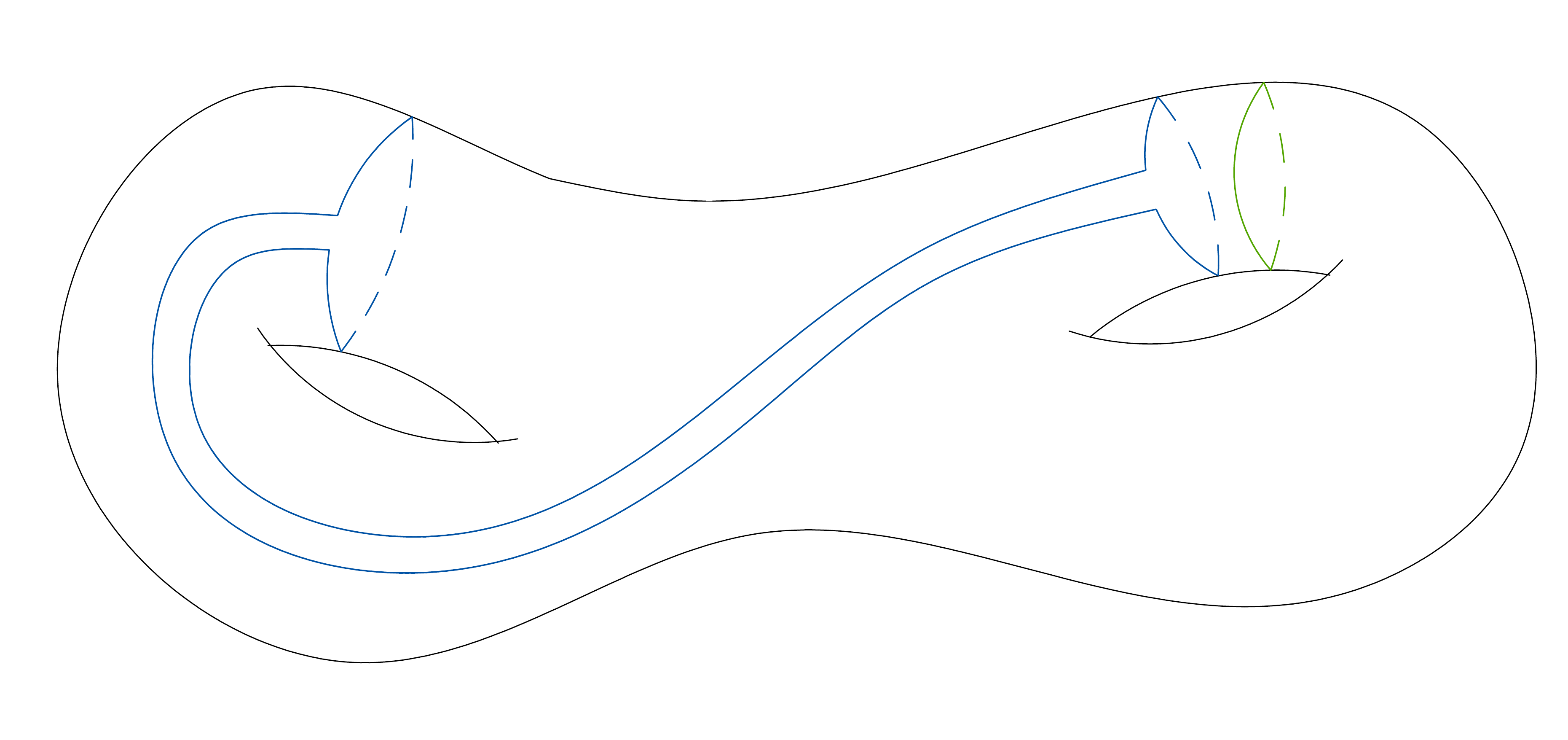}};

		\draw [->] (6.06,0.98) to (5.7,1.16); 
		\draw [->] (9.845,1.1) to (9.4,1); 
		\draw [->] (8.85,0.44) to (9.1,0.15); 
		\draw [->] (7.67,-0.23) to (7.37,0.1); 	
		
		\draw (9.0,-0.4) node {$N_j \#_\gamma P_1(N_i)$}; 
		\draw (10.8,0.9) node {$N_i$}; 
		
		\draw (8.4,-1.8) node {$\N^{ij}=(N_j\#_\gamma P_1(N_i),N_i)$}; 
		

	\end{tikzpicture}
	
	\caption{Operation (H3) which transforms $\N=(N_j,N_i)$ into $\N^{ij} =(N_j\#_\gamma P_1(N_i),N_i)$. }\label{figure:H3_operation}
	
\end{figure}

In the same way operations (H1) and (H2) are analogues of (T1) and (T2): 
$$
\varphi_{\N^\sigma} = n_\sigma \circ \varphi_\N \text{\ \ \  and \ \ \ } \varphi_{\N^i} = n_i \circ \varphi_\N.
$$

Since elementary Nielsen transformations generate $\aut(F_r)$, we have the following straightforward \mbox{conclusion.}

\begin{proposition}\label{propostion:equivalence_as_strong_equivalence_plus_operations_H1-H3}
	Two epimorphisms $\varphi_\N$ and $\varphi_{\N'}$ induced by $\N,\N' \in \H_r(M)$ are equivalent if and only if $\N'$ can be transformed by using a finite number of operations (H1) -- (H3) to a system $\N''$ such that $\varphi_\N$ and $\varphi_{\N''}$ are strongly equivalent. \qed
\end{proposition}

In particular, if $M$ is a manifold for which the map 
$$
\overline{\overline{\Theta}} \colon \H^{fr}_r(M)/_{\diffd(M)} \to \epi(\pi_1(M),F_r)/_\simeq
$$ is a~bijection, then $\varphi_\N$ and $\varphi_{\N'}$ are equivalent if and only if $\N''$ can be obtained to represent the same element of $\H^{fr}_r(M)/_{\diffd(M)}$ as~$\N$.

\begin{lemma}\label{lemma:H1_H3_do_not_change_orientability}
	The operations (H1)--(H3) on $\N$ do not change the orientability of $M|\N$.
\end{lemma}

\begin{proof}
	It is clear for (H1) and (H2). For (H3) if $\alpha$ is an orientation-reversing loop in $M|\N$, then $M|\N \setminus \im \alpha$ is also connected and a path $\gamma$ between $N_j$ and $P_1(N_i)$ can be taken to be disjoint from $\alpha$, so $M|\N^{ij}$ is also non-orientable by Proposition \ref{proposition:not_cobordant_if_one_complement_is_orientable_and_second_not}. If $M|\N$ is orientable, but $\alpha$ is an orientation-reversing loop in $M|\N^{ij}$, then it intersects $N_j$ and $P_1(N_i)$ in $P(\gamma)$, the tubular neighbourhood of $\gamma$. When $\alpha$ intersects $N_j$ and goes into $P(\gamma)$, it may pass through $P(\gamma)$ and $P_{[0,1]}(N_i) \cong [0,1] \times N_i$ or again intersect $N_j$. Note that $P_{[0,1]}(N_i)$ is orientable, because $P_1(N_i)$ is orientable as a submanifold of the orientable manifold $M|\N$. Thus $\alpha$ may be changed to another orientation-reversing loop lying outside $\N$, a contradiction. Therefore $M|\N^{ij}$ is also orientable.
\end{proof}

\begin{theorem}\label{theorem:equivalence_for_surfaces_Nielsen_transformations}
	Let $\Sigma$ be a closed surface and let $1 \leq r \leq \corank(\pi_1(\Sigma))$ be an integer. Denote by $p$ the number of equivalence classes of epimorphisms $\pi_1(\Sigma) \to F_r$. Then
	\begin{enumerate}[(1)]
		\item If $\Sigma$ is orientable or non-orientable of odd genus, then $p=1$.
		\item if $\Sigma=S_{2m}$ is non-orientable of genus $2m$, then
		\begin{itemize}
			\item if $r<m$, then $p=2$,
			\item if $r=m$, then $p=1$.
		\end{itemize}
	\end{enumerate}	
\end{theorem}

\begin{proof}
	For the first part note that $1 \leq p \leq q$, where $q$ is the number of strong equivalence classes of epimorphisms $\pi_1(\Sigma)\to F_r$, and $q=1$ if $\Sigma$ is orientable or non-orientable of odd genus. If $\Sigma$ is non-orientable of genus $2m$, then by Theorem \ref{theorem:calculations_of_cobordisms_for_surfaces} and Proposition~\ref{propostion:equivalence_as_strong_equivalence_plus_operations_H1-H3} we need to investigate the operations (H1)--(H3) on the systems $\N_0$ and $\N_I$ for $\varnothing \neq I \subset \{1,\ldots,r\}$. Since by the above lemma the operations do not change the orientability of complements of systems, $\varphi_{\N_0}$ and $\varphi_{\N_I}$ cannot be equivalent for any $I$, so $p \geq 2$ if $r<m$. We will show that all $\N_I$ induce equivalent epimorphisms.
	
	Use the operation (H3) on $\N_J = \N = (N_1,\ldots,N_r)$, for $i\notin J$ and $j\in J$, obtaining the system $\N^{ij}$ which represents the same element in $\H^{fr}_r(\Sigma)/_{\diffd(\Sigma)}$ as $\N_I$ for some $I$. We will show that $I= J \cup \{i\}$.
	
	First, note that $l \in J$ if and only if $\Sigma | \N \cup P(N_l)$ is non-orientable. Let us divide the proof into four steps:
	\begin{itemize}
		\item $j \in I$: It follows from the fact that $\Sigma |\N^{ij} \cup P(N_j \#_\gamma P_1(N_i)) = \Sigma |\N \cup P(N_j)$ is non-orientable, since $j\in J$.
		
		\item $J\setminus\{j\} \subset I$: Let $l \in J\setminus\{j\}$. Thus there is an orientation-reversing loop $\alpha$ in $\Sigma|\N\cup P(N_l)$ which intersects $N_l$ in a single point. Since a tubular neighbourhood of $\alpha$ is a M\"obious band, $\Sigma|\N \setminus \im\alpha$ is also connected and $\gamma$ in (H3) can be taken to be disjoint from $\alpha$. Thus $\im\alpha \subset \Sigma|\N^{ij}\cup P(N_l)$, so the latter subspace is non-orientable. Therefore $l\in I$ since (H3) does not depend on the choice of $\gamma$ up to framed cobordism.
		
		\item $i\in I$: Take an orientation-reversing loop $\alpha$ in $\Sigma|\N \cup P(N_j)$ intersecting $N_j$ in a~single point $x$, which is a starting point of an arc $\gamma$ joining $N_j$ with $P_1(N_i)$, and intersecting $\partial P(\gamma)$ in a single point $y$. Thus we may write $\alpha = \alpha_1 \cdot \alpha_2$, where $\alpha_1$ is a path outside $P(\gamma)$ joining $y$ with $x$. Let the endpoint of $\gamma$ in $P_1(N_i)$ correspond to $(1,z) \in \{1\} \times N_i$ and take a path $\tau \colon [-1,1] \to P(N_i) \cong [-1,1]\times N_i$ defined by $\tau(t) = (-t,z)$. Moreover, take a path $\beta$ from $\tau(1)\in P_{-1}(N_i)$ to $y$, which is contained in $\Sigma |\N \setminus \im \gamma$ (such a path exists since $\gamma$ does not disconnect $\Sigma|\N$). Now, form the~loop $\alpha' = \alpha_1 \cdot \gamma \cdot \tau\cdot \beta$, which is contained in $\Sigma |\N^{ij} \cup P_{[-1,\frac{1}{2}]}(N_i)$ by taking a smaller tubular neighbourhood of $\gamma$ used in the connected sum $N_j \#_\gamma P_1(N_i)$. Remember that the neighbourhood of $N_i$ should be considered smaller than $P(N_i)$, say $P_{[-1,\frac{1}{2}]}(N_i)$, after performing (H3) operation. The loop $\alpha'$ is orientation-reversing since it is homotopic to $\alpha \cdot \overline{\alpha_2} \cdot \gamma\cdot\tau \cdot \beta$, where $\overline{\alpha_2}$ is the inverse path for $\alpha_2$, and $ \overline{\alpha_2} \cdot \gamma\cdot\tau \cdot \beta$ is orientation-preserving as it can be homotoped to lie in $\Sigma|\N \cup P(N_i)$, which is orientable. Therefore $i \in I$.
		
		\item $l \notin I$ if $l \notin J\cup \{i\}$: if $\Sigma|\N^{ij}\cup P(N_l)$ contains an orientation-reversing loop, then as in Lemma \ref{lemma:H1_H3_do_not_change_orientability} it leads to an orientation-reversing loop in $\Sigma|N\cup P(N_l)$, a contradiction.
		
	\end{itemize}
	Thus using (H3) we may transform any $\N_J$ to be the same element of $\H^{fr}_r(\Sigma)/_{\diffd(\Sigma)}$ as $\N_{\{1,\ldots,r\}}$, so they all induce equivalent epimorphisms.
\end{proof}

\section{Reeb graphs and Reeb epimorphisms}\label{section:reeb_epimorphisms}

In this section we establish relations between Reeb graph theory and the earlier results. First, let us state basic notions.

A \emph{smooth triad} is a triple $(W,W_-,W_+)$, where $W$ is a manifold with boundary $\partial W = W_- \sqcup W_+$ (possibly $W_\pm = \varnothing$). A smooth function $f\colon W \to [a,b]$ is a \emph{function on a smooth triad} $(W,W_-,W_+)$ if $f^{-1}(a) = W_-$, $f^{-1}(b)=W_+$ and all its critical points are contained in $\int W$, the interior of $W$.

Let $f\colon W \to \R$ be a function with finitely many critical points on a~smooth triad $(W,W_-,W_+)$. We say that $x,y \in W$ are in \emph{Reeb relation} $\sim_\reeb$ if and only if $x$ and $y$ belong to the same connected component of a~level set
of~$f$. The quotient space $W/\!\sim_{\mathcal{R}}$ is denoted by
$\reeb{(f)}$ and called the \emph{Reeb graph} of $f$.

The Reeb graph of the function $f$ as above is homeomorphic to a
finite graph, i.e. to a one-dimensional finite CW-complex (see
\cite{Reeb}, \cite{Sharko}). The vertices of~$\reeb{(f)}$
correspond to the components of $W_\pm$ and to the components of
level sets of $f$ containing critical points. The homomorphism $q_\# \colon \pi_1(W) \to \pi_1(\reeb(f)) \cong F_r$ induced by the quotient map $q \colon W \to \reeb(f)$ is surjective (see \cite{KMS}) and is called the \emph{Reeb epimorphism} of~$f$. The number $r$ as above is called the \emph{cycle rank} of $\reeb(f)$ and it is equal to the first Betti number $\beta_1(\reeb(f))$.

For an oriented graph (i.e. each its edge has a chosen direction), the \emph{indegree} and \emph{outdegree} of its vertex $v$ are the number of incident edges which income and outgo from $v$, respectively. The \emph{degree} $\deg(v)$ is the numbers of all incident edges to $v$. 

By the quotient topology $f$ induces the continuous function $\overline{f}\colon\reeb{(f)} \to \R$ such that $f = \overline{f}\circ q$. It is strictly monotonic on the edges of $\reeb(f)$ and has extrema only at vertices of degree $1$. A function on a graph satisfying these properties induces an orientation of the graph called a \emph{good orientation} (see Sharko \cite{Sharko}). Thus any Reeb graph is considered with a good orientation. 
Orientations of graphs presented in figures in this paper are from the bottom to the top.

Recall that $f\colon W \to \R$ is a \emph{Morse function} if it is smooth and all its critical points are non-degenerate. A~Morse function $f$ is \emph{simple} if its critical levels contain only one critical point and it is \emph{ordered} if for any two critical points $p$ and $p'$ of $f$ if $\ind(p) < \ind(p')$, then $f(p) < f(p')$, where $\ind(p)$ is the index of non-degenerate critical point $p$. Note that the Reeb graph of an ordered Morse function on a manifold of dimension at least~$3$ is a tree. The same is also true for a surface and self-indexing Morse function $f\colon W\to \R$, i.e. $f(p) = \ind(p)$ for each critical point~$p$ (see \cite{Michalak-DCG}).

For some reasons it is convenient to use a slightly different condition on a Morse function than its simplicity. Namely, we say that a Morse function $f$ is \emph{$\reeb$-simple} if any vertex of $\reeb(f)$ corresponds exactly to a one critical point of $f$. Equivalently, any connected component of a level set of $f$ contains at most one critical point. Obviously, a simple Morse function is $\reeb$-simple. Moreover, having an $\reeb$-simple Morse function we can easily make a small perturbation to get a simple Morse function without changing its Reeb graph.

Let $f\colon W \to \R$ be an $\reeb$-simple Morse function on the triad $(W,W_-,W_+)$ of dimension $n=\dim W$. The vertices of degree~$1$ in $\reeb(f)$ correspond to components of $\partial W$ and to the extrema of $f$, the critical points of indices $0$ and $n$. If $W$ is an orientable surface, then all other vertices have degree $3$ (see \cite{Michalak-TMNA}). However, if $W$ is not an orientable surface, then the other vertices of $\reeb(f)$ have degrees $2$ and $3$. In addition, for $n\geq 3$ vertices of degree $3$ correspond to critical points of index $1$ (with indegree $2$) or index $n-1$ (with outdegree $2$), and vertices of degree $2$ correspond to indices $1,\ldots,n-1$ (see~\cite{Michalak-TMNA, Michalak-DCG, Reeb}).

In \cite{Michalak-DCG} there were defined \emph{combinatorial modifications} of Reeb graphs of simple (or $\reeb$-simple) Morse functions numbered by (1) -- (9), (11), (12), and the modification (10) transforming the Reeb graph of a simple Morse function to the Reeb graph of a Morse function, which accumulates some critical points to the one connected component of level set. We will use them extensively in the proofs. They were introduced for manifolds of dimension at least $3$, but they can be well-defined for also non-orientable surfaces because of the same degree-index correspondence as we mentioned above (with some caution for modification (7)). For orientable surfaces we need to omit modifications with vertices of degree $2$. In fact, for orientable surfaces Fabio and Landi \cite{Fabio-Landi} introduced such operations called \emph{elementary deformations}. In order not to separate these cases, we will use the term combinatorial modification for Reeb graphs of functions on any manifold, having in mind that for orientable surfaces there are no modifications with vertices of degree~$2$. Moreover, the modification (6) for orientable surfaces works in both ways since it deals with critical points of the same index $1$ (it corresponds to elementary deformations ($\rm{K}_2$)~and~($\rm{K}_3$)~of~\cite{Fabio-Landi}).

The considerations in this section are motivated by and based on the following theorem of the second author \cite{Michalak-DCG} which summarized previous results.

\begin{theorem}[{\cite[Theorem 5.2]{Michalak-DCG}}]\label{thm:equivalent_conditions}
	Let $M$ be a closed, smooth and connected manifold of dimension at least two. The following are equivalent:
	\begin{enumerate}[(a)]
		\item There exists a Morse function $f\colon M \rightarrow \R$ (simple if $M$ is not an orientable surface) such that $\beta_1(\reeb(f)) = r$.
		\item There exists an epimorphism $\pi_1(M) \rightarrow F_r$.
		\item There exists a regular and independent system of hypersurfaces of size $r$ in $M$.
	\end{enumerate}
\end{theorem}

\begin{remark}
	The equivalence of conditions (b) and (c)  has been described by
	Cornea \cite[Theorem 1]{Cornea} and for combinatorial manifolds by
	Jaco \cite[Theorem~2.1]{Jaco}. It is evident that the condition (a)
	implies the conditions (b) and (c) (see \cite{KMS}). Moreover, Gelbukh \cite{Gelbukh:DCG} showed the equivalence of conditions (a) and (b) for orientable manifolds by using foliation theory. In \cite[Theorem 5.2]{Michalak-DCG} there is a crucial proof that (c) implies (a). Theorems~\ref{theorem:factorization_by_Reeb_epimorphism}~and~\ref{theorem:epimorphism_is_induced_by_regular_and_independent_system} provide a direct proof that (b) or (c) imply~(a).
\end{remark}

\begin{remark}
	Orientable surfaces have a unique property that a simple (and also $\reeb$-simple) Morse function on a triad of orientable surface of genus $g$ has the Reeb graph with cycle rank equal to $g$ (see \cite{Edelsbrunner} and also \cite{Michalak-TMNA}).
\end{remark}

The combinatorial modifications of Reeb graphs together with the construction of the initial graph (see Figure \ref{figure:initial_graph}) are the main ingredients in the following solution of realization problem for Reeb graphs from \cite{Michalak-DCG}. It was a natural problem to determine which graph can be the Reeb graph of a smooth function with isolated critical points on a given manifold (cf. \cite{Saeki_Reeb_spaces}, see~Remark~\ref{remark:Saeki}).

\begin{proposition}[{\cite{Michalak-DCG}}]
	For any finite graph $\Gamma$ with good orientation there is a finite sequence of combinatorial modifications (1)~--~(12) which transform the initial graph to $\Gamma$ up to vertices of degree $2$.
\end{proposition}

\begin{theorem}[{\cite[Theorem 6.4]{Michalak-DCG}}]\label{theorem:realization-DCG}
	Let $M$ be a closed, connected manifold of dimension $n\geq 2$ and $\Gamma$ be a finite connected oriented graph. There exists a Morse function $f\colon M \to \R$ such that $\reeb(f)$ is orientation-preserving homeomorphic to~$\Gamma$ if and only if $\Gamma$ has a good orientation and $\beta_1(\Gamma) \leq \corank(\pi_1(M))$. Moreover, if $M$ is not an orientable surface and the maximum degree of a vertex in $\Gamma$ is not greater than $3$, then $f$ can be taken to be simple. 
\end{theorem}

In this section we resolve the realization problem for a manifold with boundary together with representation of an epimorphism as the Reeb epimorphism of a Morse function. We do it by constructing a Morse function with certain components of level sets prescribed by a system of hypersurfaces.


\subsection{The initial graph and factorization by Reeb epimorphism}

Recall that $\pi_0(X)$ is the set of path components of a space $X$ and their number is denoted by $|\pi_0(X)|$. Since all discussed by us spaces are locally path connected, the set $\pi_0(X)$ of path components of a space $X$ is equal to  the set  of connected components~of~$X$.

We call a graph $\Gamma$ \emph{admissible} for a triad $(W,W_-,W_+)$ if it has at least $|\pi_0(W_-)|$ vertices of degree $1$ and indegree $0$ and at least $|\pi_0(W_+)|$ vertices of degree $1$ and outdegree $0$.

The graph presented in Figure \ref{figure:initial_graph} (a) is called the \emph{initial graph} (with a given cycle rank $r$). Recall that by our convention it has orientation from the bottom to the top. 
We distinguish a spanning tree in the initial graph, coloured red in the figure. Moreover, we order the edges $e_1,\ldots,e_r$ outside the tree. We define also a version of the initial graph which is admissible for a triad $(W,W_-,W_+)$ (see Figure \ref{figure:initial_graph} (b)).

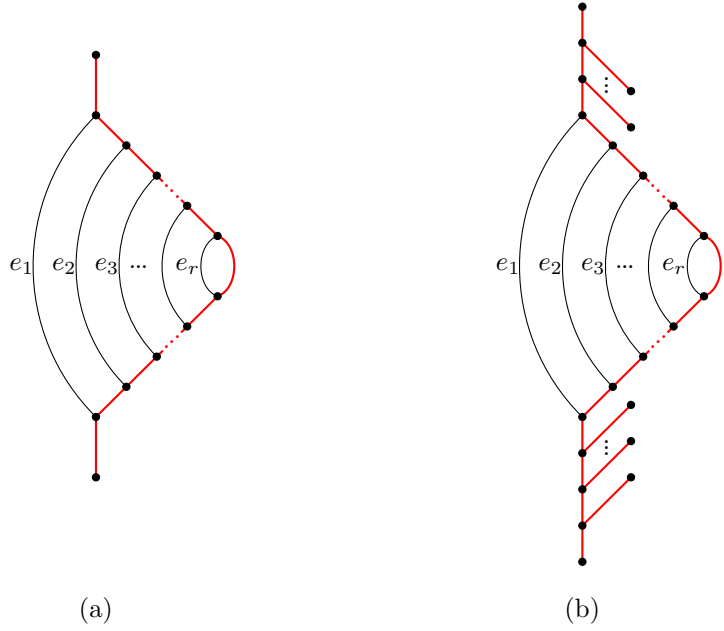
\begin{figure}[h]
	\centering
	\begin{tikzpicture}[scale=0.8]

		
		\draw[thick,red] (6,0) -- (6,1);
		\draw[thick,red] (6.5,1.5) -- (6,1);
		\draw[thick,red] (6.5,1.5) -- (7,2);
		\draw[thick,red] (7.1,2.1) -- (7,2);
		\draw[thick,red] (7.4,2.4) -- (7.5,2.5);
		\draw[thick,red] (7.5,2.5) -- (8,3);

		\draw[thick,red] (6,7) -- (6,6);
		\draw[thick,red] (6.5,5.5) -- (6,6);
		\draw[thick,red] (6.5,5.5) -- (7,5);
		\draw[thick,red] (7.1,4.9) -- (7,5);
		\draw[thick,red] (7.4,4.6) -- (7.5,4.5);
		\draw[thick,red] (7.5,4.5) -- (8,4);

		\draw (8,4) to[out=200,in=160] (8,3);
		\draw[thick,red] (8,4) to[out=340,in=20] (8,3);

		\draw (7.5,4.5) to[out=225,in=135] (7.5,2.5);
		\draw (7,5) to[out=225,in=135] (7,2);
		\draw (6.5,5.5) to[out=225,in=135] (6.5,1.5);
		\draw (6,6) to[out=225,in=135] (6,1);

		\draw[red] (7.34,2.34) node { .};
		\draw[red] (7.25,2.25) node { .};
		\draw[red] (7.16,2.16) node { .};

		\draw[red] (7.34,4.66) node { .};
		\draw[red] (7.25,4.75) node { .};
		\draw[red] (7.16,4.84) node { .};

		\filldraw (6,0) circle (1.7pt);
		\filldraw (6,1) circle (1.7pt);
		\filldraw (6.5,1.5) circle (1.7pt);
		\filldraw (7,2) circle (1.7pt);
		\filldraw (7.5,2.5) circle (1.7pt);
		\filldraw (8,3) circle (1.7pt);
		\filldraw (8,4) circle (1.7pt);
		\filldraw (7.5,4.5) circle (1.7pt);
		\filldraw (7,5) circle (1.7pt);
		\filldraw (6.5,5.5) circle (1.7pt);
		\filldraw (6,6) circle (1.7pt);
		\filldraw (6,7) circle (1.7pt);

		\draw (4.73,3.5) node {$e_1$};
		\draw (5.43,3.5) node {$e_2$};
		\draw (6.13,3.5) node {$e_3$};
		
		\draw (6.6,3.5) node {.};
		\draw (6.7,3.5) node {.};
		\draw (6.8,3.5) node {.};
		
		\draw (7.47,3.5) node {$e_r$};
		
		\draw (6,-2.2) node{(a)};

		
		\draw[thick,red] (14,-1.4) -- (14,1);
		\draw[thick,red] (14.5,1.5) -- (14,1);
		\draw[thick,red] (14.5,1.5) -- (15,2);
		\draw[thick,red] (15.1,2.1) -- (15,2);
		\draw[thick,red] (15.4,2.4) -- (15.5,2.5);
		\draw[thick,red] (15.5,2.5) -- (16,3);
		
		
		\draw[thick,red] (14,-0.8) -- (14.8,0);
		\draw[thick,red] (14,-0.2) -- (14.8,0.6);
		\draw[thick,red] (14,0.4) -- (14.8,1.2);
		
		\draw (14.4,0.4) node {.};
		\draw (14.4,0.5) node {.};
		\draw (14.4,0.6) node {.};

		\draw[thick,red] (14,7.8) -- (14,6);
		\draw[thick,red] (14.5,5.5) -- (14,6);
		\draw[thick,red] (14.5,5.5) -- (15,5);
		\draw[thick,red] (15.1,4.9) -- (15,5);
		\draw[thick,red] (15.4,4.6) -- (15.5,4.5);
		\draw[thick,red] (15.5,4.5) -- (16,4);

		
		\draw[thick,red] (14,6.6) -- (14.8,5.8);
		\draw[thick,red] (14,7.2) -- (14.8,6.4);
		
		\draw (14.4,6.4) node {.};
		\draw (14.4,6.5) node {.};
		\draw (14.4,6.6) node {.};

		\draw (16,4) to[out=200,in=160] (16,3);
		\draw[thick,red] (16,4) to[out=340,in=20] (16,3);

		\draw (15.5,4.5) to[out=225,in=135] (15.5,2.5);
		\draw (15,5) to[out=225,in=135] (15,2);
		\draw (14.5,5.5) to[out=225,in=135] (14.5,1.5);
		\draw (14,6) to[out=225,in=135] (14,1);

		\draw[red] (15.34,2.34) node { .};
		\draw[red] (15.25,2.25) node { .};
		\draw[red] (15.16,2.16) node { .};

		\draw[red] (15.34,4.66) node { .};
		\draw[red] (15.25,4.75) node { .};
		\draw[red] (15.16,4.84) node { .};

		\filldraw (14,-1.4) circle (1.7pt);
		\filldraw (14,-0.8) circle (1.7pt);
		\filldraw (14,-0.2) circle (1.7pt);
		\filldraw (14.8,0) circle (1.7pt);
		\filldraw (14.8,0.6) circle (1.7pt);
		\filldraw (14.8,1.2) circle (1.7pt);
		\filldraw (14,0.4) circle (1.7pt);
		\filldraw (14,1) circle (1.7pt);
		\filldraw (14.5,1.5) circle (1.7pt);
		\filldraw (15,2) circle (1.7pt);
		\filldraw (15.5,2.5) circle (1.7pt);
		\filldraw (16,3) circle (1.7pt);
		\filldraw (16,4) circle (1.7pt);
		\filldraw (15.5,4.5) circle (1.7pt);
		\filldraw (15,5) circle (1.7pt);
		\filldraw (14.5,5.5) circle (1.7pt);
		\filldraw (14,6) circle (1.7pt);
		\filldraw (14,6.6) circle (1.7pt);
		\filldraw (14,7.2) circle (1.7pt);
		\filldraw (14,7.8) circle (1.7pt);
		\filldraw (14.8,6.4) circle (1.7pt);
		\filldraw (14.8,5.8) circle (1.7pt);

		\draw (12.73,3.5) node {$e_1$};
		\draw (13.43,3.5) node {$e_2$};
		\draw (14.13,3.5) node {$e_3$};
		
		\draw (14.6,3.5) node {.};
		\draw (14.7,3.5) node {.};
		\draw (14.8,3.5) node {.};
		
		\draw (15.48,3.5) node {$e_r$};
		
		\draw (14,-2.2) node{(b)};

	\end{tikzpicture}
	
	\caption{ (a) the initial graph with distinguished tree and ordered edges outside it; (b) the initial graph admissible for $(W,W_-,W_+)$.}\label{figure:initial_graph}
\end{figure}

The initial graph with cycle rank equal to $g$ occurs easily as the Reeb graph of a height function on an orientable surface of genus $g$. In fact, by \cite[Theorem 5.6]{Michalak-TMNA} it~can be the Reeb graph of a Morse function on any closed surface with Euler characteristic at most $2-2g$.

\begin{theorem}\label{theorem:factorization_by_Reeb_epimorphism}
	Let $\partial W = W_- \sqcup W_+$ and $\mathcal{N}=(N_1,\ldots,N_r)$ be a system of hypersurfaces without boundary in~$W$. Then the induced homomorphism $\varphi_\N$ is factorized by the Reeb epimorphism of a simple Morse function $f\colon W \to \R$ on the smooth triad $(W,W_-,W_+)$ such that the connected components of $\N$ are components of some regular level $f^{-1}(c) = \N \sqcup V$, where $V$ is a non-empty submanifold.
	Moreover, if we allow $f$ to be not necessarily simple for $\dim W =2$, then
	we can construct $f$ in such a way that
	$$\beta_1(\reeb(f)) = |\pi_0(\N)| - |\pi_0(W|\N)| +1.$$
	Furthermore, if $\N$ is regular and independent, then $\reeb(f)$ can be taken to be homeomorphic to the initial graph admissible for $(W,W_-,W_+)$, with cycle rank equal to $r$ and such that $N_i$ corresponds to the edge $e_i$ for each $i$ as in Figure \ref{figure:initial_graph}.
\end{theorem}

\begin{proof}
	The construction of the desired function is analogous to that in Theorem \ref{thm:equivalent_conditions} $(c) \implies (a)$ in \cite{Michalak-DCG}, but here $\mathcal{N}$ may not be independent and regular. For details on the existence of Morse functions and gluing operations we refer to \cite{Milnor}.
	
	Take an ordered Morse function $h\colon W|\N \to \R$ on the triad $$(W|\N, P_{-1}(\N) \sqcup W_+,P_{1}(\N)\sqcup W_-)$$ and a regular value $d$ from \cite[Lemma~3.3]{Michalak-DCG} such that $V:= h^{-1}(d)$ has the same number of connected components as $W|\N$. Let
	$$P(V) := h^{-1}([d-\varepsilon,d+\varepsilon]) \cong V \times [-1,1],$$
	$$Q_- := h^{-1}((-\infty,d-\varepsilon]) \ \ \text{ and }\ \  Q_+ := h^{-1}([d+\varepsilon,\infty)).$$
	Then
	$$
	\partial Q_- = P_{-1}(V) \sqcup P_{-1}(\N)\sqcup W_+ \ \   \text{ and }\ \  \partial Q_+ = P_{1}(V) \sqcup P_{1}(\N)\sqcup W_-.
	$$
	Thus $V$, $Q_+$ and $Q_-$ have the same number of connected components. 
	Now, take simple and ordered Morse functions 
	$$g_- \colon Q_- \rightarrow [-2,-1] \text{ on the triad } (Q_-,\emptyset,\partial Q_-),$$
	$$g_+ \colon Q_+ \rightarrow [1,2] \text{ on the triad } (Q_+,\partial Q_+,\emptyset).$$
	
	Let us glue them together with suitable projections $P(\N\sqcup V) \to [-1,1]$ obtaining the simple Morse function $f \colon W \to \R$ with regular value $0$ such that $f^{-1}(0) = \N \sqcup V$.
	
	Let $q\colon W \to \reeb(f)$ and $g \colon \reeb(f)\to \reeb(f)/q(W|\N) = \bigvee_{i=1}^r \es^1_i$ be the quotient maps. The map $g$ sends $q(W|\N)$ to the basepoint and $q(P(N_i)) \cong [-1,1]$ linearly and orientation-preserving onto $\es^1_i$. It is clear that $\varphi_\N = (f_\N)_\#= (g\circ q)_\# = g_\# \circ q_\#$, so $\varphi_\N$ is factorized by the Reeb epimorphism $q_\#$ of $f$.
	
	Now, let us compute $\beta_1(\reeb(f))$. The subset $q(W|\N)$ of $\reeb(f)$ is homeomorphic to the Reeb graph $\reeb(f|_{W|\N})$, so it has $|\pi_0(W|\N)|$ connected components. If $\dim W \geq 3$, then the components of $\reeb(f|_{W|\N})$ are trees, because components of $\reeb(f|_{Q_\pm})$ are trees by \cite[Proposition 3.2]{Michalak-DCG} (since the Morse functions on $Q_\pm$ are ordered)
	and they are gluing through $\reeb(f|_{P(V)}) \cong [-1,1]$, one component of $\reeb(f|_{Q_\pm})$ with only one component of $\reeb(f|_{P(V)})$.
	In the case of surfaces by the same fact
	we may define Morse functions (self-indexing, not simple) on components of $Q_\pm$ whose Reeb graphs are trees.
	In both the cases $q(W|\N)$ has $|\pi_0(W|\N)|$ components which are trees, and so they are contractible.
	
	Thus the quotient $\reeb(f)/q(W|\N)$ can be obtained from $\reeb(f)$ by first the contraction of components of $q(W|\N)$, and then by gluing them to the one point. The first operation does not change the first Betti number, but the second increases it by one for each gluing of two points. Hence  $\reeb(f)/q(W|\N)$ has cycle rank equal to $\beta_1(\reeb(f)) + |\pi_0(W|\N)| -1$. On the other hand, it is clear that $\reeb(f)/q(W|\N)$ is homeomorphic to the wedge product of $|\pi_0(\N)|$ circles. Therefore
	$$
	|\pi_0(\N)| = \beta_1(\reeb(f)) + |\pi_0(W|N)| -1.
	$$
	
	Now, let $\mathcal{N}$ be a regular and independent system of hypersurfaces. Let $\dim W \geq 3$. Since $W|\N$ and $V$ are connected, the manifolds $Q_\pm$ are also connected, so we may assume that $g_\pm$ has only one critical point being extremum. Then by \mbox{\cite[Proposition~3.2]{Michalak-DCG}} the Reeb graph $\reeb(g_-)$ (resp. $\reeb(g_+)$) is a~tree with one minimum (maximum) and so all vertices of degree $3$ have indegree~$1$ (outdegree $1$). By the proved formula on cycle rank we have $\beta_1(\reeb(f)) = r$. We proceed as in the proof of \cite[Proposition~6.2]{Michalak-DCG}. By means of the combinatorial modifications we move up (move down) all vertices of degree $2$ in $\reeb(g_-)$ (in $\reeb(g_+)$ respectively). 
	Then by using the modification (4) on $f|_{Q_-} = g_-$ and (5) on $f|_{Q_+} = g_+$  we can obtain a simple Morse function on $W$ whose Reeb graph is homeomorphic to the initial graph admissible for $(W,W_-,W_+)$ with the desired correspondence between $N_i$ and its edges $e_i$.	
	
	The last statement in the case of surfaces can be similarly obtained with some additional effort or it follows by Theorem~\ref{theorem:realization_of_graph_for_surfaces_with_prescribed_level_sets} (Theorem \ref{theorem:factorization_by_Reeb_epimorphism} is not used in the proof of Theorem~\ref{theorem:realization_of_graph_for_surfaces_with_prescribed_level_sets} for surfaces).
\end{proof}

\begin{remark}
	For $\dim W =2$ note that the components of $Q_\pm$ can be either orientable or non-orientable, even if $W$ is non-orientable. Since on an orientable component a simple Morse function has a Reeb graph with the maximum cycle rank, it is a reason why simplicity of a Morse function in the theorem is excluded also if $W$ is a non-orientable surface.
\end{remark}

Using the above theorem we may easily prove the last part of Proposition~\ref{proposition:epi_and_regular_is_independent}. If $\N$ is regular and $\varphi_\N$ is surjective, then by Theorem \ref{theorem:factorization_by_Reeb_epimorphism} the induced epimorphism $\varphi_\N$ is factorized by a Reeb epimorphism of rank $r' =r -|\pi_0(W|\N)| +1$. Since $r \leq r'$, it implies that $|\pi_0(W|\N)| \leq 1$, so $\N$ is independent.

This theorem can be also used to easily prove the following known fact for orientable surfaces.

\begin{corollary}   \label{corollary:orientable_surface_maximal_epi_has_rank_g}
	Any epimorphism $\varphi\colon \pi_1(\Sigma_g) \to F_r$ is factorized through an epimorphism $\pi_1(\Sigma_g) \to F_g$.
\end{corollary}

\begin{proof}
	By Theorem \ref{theorem:factorization_by_Reeb_epimorphism} $\varphi$ is factorized by the Reeb epimorphism of a~simple Morse function $f \colon \Sigma_g \to \R$, whose rank is equal to $\beta_1(\reeb(f))$. By \cite{Edelsbrunner} (cf. \cite{Michalak-TMNA}) the Reeb graph of a simple Morse function on $\Sigma_g$ has always cycle rank equal to $g$, so $\beta_1(\reeb(f))=g$.
\end{proof}

In fact, since any two epimorphisms $\pi_1(\Sigma_g) \to F_g$ are strongly equivalent by Theorem \ref{thm:grigorchuk}, for a fixed $\psi\colon \pi_1(\Sigma_g) \to F_g$ any $\varphi\colon \pi_1(\Sigma_g) \to F_r$ is factorized through $\psi \circ \eta$ for some $\eta \in \aut(\pi_1(\Sigma_g))$.

\subsection{Reeb number of manifolds with boundary}\label{Reeb number of manifolds with boundary}

The \emph{Reeb number} $\reeb(M)$ of a closed manifold $M$ was an object of studies in \cite{Michalak-TMNA,Michalak-DCG} and without using this name in \cite{Edelsbrunner, Gelbukh:DCG, Gelbukh:filomat}.
It is defined as the maximum cycle rank among Reeb graphs of functions with finitely many critical points on $M$. By Theorem \ref{thm:equivalent_conditions} (\cite[Theorem 5.2]{Michalak-DCG}) and \cite[Lemma 3.5]{Michalak-TMNA} the equality $\reeb(M) = \corank(\pi_1(M))$ holds (see also~\cite{Gelbukh:filomat}).

For a compact manifold $W$, possibly with boundary, we define (following Cornea \cite{Cornea}) the number $C(W)$ to be the maximum number of connected components in a proper,
$2$-sided submanifold $N$ of $W$ such that $W\setminus N$ is connected.
In other words, it is the maximum size of an independent and regular system in $W$. It is clear by Theorem \ref{thm:equivalent_conditions}  that $\reeb(M) = C(M)$ for a closed manifold $M$.

The following fact was proven by Jaco \cite{Jaco} for
combinatorial manifolds. Cornea announced only  inequality $C(W)
\geq \corank(\pi_1(W)) - |\pi_0(\partial W)| +1$ if $\partial W \neq
\varnothing$, but the theorem holds also in the smooth category.

\begin{theorem}\label{theorem:Jaco for smooth}
	$C(W) = \corank(\pi_1(W))$.
\end{theorem}

\begin{proof}
	If there is an independent and regular system $\mathcal{N}$ of size $k = C(W)$, then the induced homomorphism $\varphi_\N$ is onto $F_k$, so $C(W) \leq \corank(\pi_1(W))$. From the other side,
	any epimorphism onto the free group of rank equal to $\corank(\pi_1(W))$ is by Theorem \ref{theorem:epimorphism_is_induced_by_regular_and_independent_system} induced by a regular and independent system,
	so $C(W) = \corank(\pi_1(W))$.
\end{proof}

Now, we extend the definition of Reeb number on manifolds with boundary. First, define $\reeb(W,W_-,W_+)$, where $\partial W = W_- \sqcup W_+$, as the maximum cycle rank among all
Reeb graphs of smooth functions with finitely many critical points on the smooth triad $(W,W_-,W_+)$. By \cite[Lemma 3.5]{Michalak-TMNA} applied for smooth triads it is attainable by simple Morse functions.

\begin{proposition}\label{proposition:reeb_number_for_triads_and_systems_without_boundary}
	$\reeb(W,W_-,W_+)$ is equal to the maximum size of an independent and regular system without boundary in $W$. Thus it does not depend on the partition $\partial W = W_- \sqcup W_+$.
\end{proposition}

\begin{proof}
	By Theorem \ref{theorem:factorization_by_Reeb_epimorphism} if $\N$ is a regular and independent system without boundary in $W$ of size $r$, then $\varphi_\N$ is factorized by the Reeb epimorphism of rank ${|\pi_0(\N)| - |\pi_0(W|\N)| +1= r}$, which gives inequality in one way. However, if $f$ is a simple Morse function on $(W,W_-,W_+)$ such that $\reeb(f)$ has cycle rank equal to $\reeb(W,W_-,W_+)$, then components of level sets of $f$ corresponding to edges of $\reeb(f)$ outside some spanning tree form a regular and independent system of hypersurfaces in $W$ of size $\reeb(W,W_-,W_+)$.
\end{proof}

\begin{definition}
	For a compact manifold $W$ with boundary we define its \emph{Reeb number} as $\reeb(W) := \reeb(W,W_-,W_+)$ for any $\partial W = W_- \sqcup W_+$.
\end{definition}

\noindent It is obvious that $\reeb(W) \leq C(W)$.

\begin{remark}
	Note that $\reeb(W)$ can be defined as the maximum cycle rank among Reeb graphs of Morse functions on $W$ which are constant on connected components of $\partial W$. We use triads for simplifying considerations.
\end{remark}

Let $\cone(X) := X \times [0,1] / X \times \{1\}$ denote the cone over a space $X$. The point corresponding to $X \times \{1\}$ is called the vertex of the cone.

For a compact manifold $W$ with a boundary divided into two parts $\partial W = A \sqcup B$, where $A_1,  \ldots,  A_k$ are all connected components of~$A$, define 
$$\cone_A(W) := W \cup_A \bigcup_{i=1}^k \cone(A_i),$$
which is obtained by gluing cones $\cone(A_i)$ and $W$ along $A$. Let $v_i$ be the vertex of $\cone(A_i)$. Clearly, we may identify $$\cone_{A}(W) \setminus \{v_1,\ldots,v_k\} \cong W \setminus A.$$

Hereafter, we denote by $\geng{\pi_1(A)}^{\pi_1(W)}$ the normal subgroup of $\pi_1(W)$ generated by all images of $\pi_1(A_i)$ in $\pi_1(W)$ by the homomorphisms induced by inclusions $A_i \subset W$.
By Seifert--van Kampen theorem
$$\pi_1(\cone_{A}(W)) \cong \pi_1(W)/\geng{\pi_1(A)}^{\pi_1(W)}.$$
It is clear that up to isomorphism this group is well-defined without referencing to the basepoint. 

The following proposition describes properties of an epimorphism onto a free group in terms of its factorization and a system of hypersurfaces, generalizing \cite[Proposition~4.2]{Stallings_corank1} of Stallings.

\begin{proposition}\label{proposition:cones-systems_omitting_part_of_boundary}
	Let $W$ be a compact manifold and  $\partial W = A \sqcup B$. Then an epimorphism $\varphi\colon \pi_1(W) \to F_r$ is factorized through $\pi_1(W)/\geng{\pi_1(A)}^{\pi_1(W)}$ if and only
	if it is induced by an independent and regular system $\N$ such that $\N \cap A = \varnothing$.
\end{proposition}

\begin{proof}
	Set $H:= \geng{\pi_1(A)}^{\pi_1(W)}$. If $\mathcal{N}$ is an independent and regular system such that $\N \cap A =\varnothing$, then clearly the images in $\pi_1(W)$ of loops in $A$ are contained in the kernel of $\varphi_\N$, so $\varphi_\N$ is factorized through $\pi_1(W)/H$.
	
	Conversely, assume that $\varphi = \psi \circ \eta$, where $\eta \colon \pi_1(W) \to \pi_1(W)/H$ and ${\psi \colon \pi_1(W)/H \to F_r}$. We proceed as in the proof of Proposition \ref{proposition:homomorphisms_are_induced_by_systems}. Let $\psi$ be induced by ${f \colon \cone_{A}(W) \to \bigvee_{i=1}^r \es^1_i}$ which is a smooth map outside $\{v_1,\ldots,v_k\}$ and the inverse image of the basepoint. Take regular values $a_i \in \es^1_i$ and define 
	$$
	N_i = f^{-1}(a_i) \subset \cone_{A}(W) \setminus \{v_1,\ldots,v_k\} \cong W \setminus A.
	$$
	Thus $\mathcal{N}=(N_1,\ldots,N_r)$ is a system in $\cone_{A}(W)$ which induces $\psi$ such that $\N \cap A =\varnothing$. Clearly, as a system in $W$ it induces $\varphi$. It is easy to check that the procedures in proofs of Lemma \ref{lemma:connected_sum_of_submanifolds} and Theorem \ref{theorem:epimorphism_is_induced_by_regular_and_independent_system} give an independent and regular system $\mathcal{N}'$ inducing $\varphi$ which also satisfies $\N'\cap A = \varnothing$.
\end{proof}

The following theorem is a generalization of Theorem \ref{thm:equivalent_conditions} from \cite{Michalak-DCG}.

\begin{theorem}\label{theorem:Reeb_epi_iff_regular_and_independent}
	For an epimorphism $\varphi \colon \pi_1(W) \to F_r$ the following are equivalent:
	
	\begin{enumerate}[(1)]
		\item $\varphi = \varphi_\N$ for an independent and regular system $\N$ without boundary;
		\item $\varphi$ is factorized through $\pi_1(W)/\langle\pi_1(\partial W)\rangle^{\pi_1(W)}$;
		\item there is a Morse function $f$ (simple if $\dim W \geq 3$) on any smooth triad $(W,W_-,W_+)$ and a spanning tree $T$ in $\reeb(f)$ such that $\varphi = (p_T \circ q)_\#$, where $q\colon W \to \reeb(f)$ and $p_T\colon \reeb(f) \to \reeb(f)/T = \bigvee_{i=1}^r \es^1$ are quotient maps.
	\end{enumerate}
	Thus $$\reeb(W) = \corank\left(\pi_1(W)/\geng{\pi_1(\partial W)}^{\pi_1(W)}\right).$$
\end{theorem}

\begin{proof}
	The equivalence of (1) and (2) follows from the above proposition for $A = \partial W$. 
	If $\varphi = \varphi_\N$ for an independent and regular system $\N=(N_1,\ldots,N_r)$ without boundary, then by Theorem \ref{theorem:factorization_by_Reeb_epimorphism} there is a Morse function $f$ (simple if $\dim W \geq 3$) on $(W,W_-,W_+)$ whose Reeb graph has cycle rank equal to $r$ and components of $\N$ are components of the same level set $f^{-1}(c)$. Thus edges corresponding to components of $\N$ are outside some spanning tree $T$ of $\reeb(f)$, and so $(p_T \circ q)_\# = \varphi_\N$. This proves that (1) implies (3), and the converse is clear. 
	
	By Proposition   \ref{proposition:reeb_number_for_triads_and_systems_without_boundary}
	we get $\reeb(W) = \corank\left(\pi_1(W)/\geng{\pi_1(\partial
		W)}^{\pi_1(W)}\right)$.
\end{proof}


\subsection{Realization of system of hypersurfaces as components of level sets of function}

For a further study of Reeb epimorphisms of Morse functions we would like to have prescribed components of level sets of the function corresponding to edges outside a spanning tree of the graph.

\begin{theorem}\label{theorem:realization_of_graph_for_surfaces_with_prescribed_level_sets}
	Let $(W,W_-,W_+)$ be a smooth triad, $W_\pm = W_1^\pm \sqcup \ldots \sqcup W_{|\pi_0(W_\pm)|}^\pm$ be a decomposition into connected components and consider a regular and independent system  $\N=(N_1,\ldots,N_r)$ of hypersurfaces without boundary in $W$. Let $\Gamma$ be a finite connected graph with good orientation, whose cycle rank is equal to $r$ and which is admissible for $(W,W_-,W_+)$. Distinguish vertices $a^\pm_1,\ldots,a^\pm_{|\pi_0(W_\pm)|}$ of degree 1 in $\Gamma$, where all $a_i^-$ have indegree $0$ and all $a_i^+$ have outdegree $0$.
	Moreover, take a spanning tree $T$ of $\Gamma$ and order the edges outside $T$ as $e_1,\ldots,e_r$.
	Then there is a Morse function $f\colon W \to \R$ on the triad $(W,W_-,W_+)$, such that $\reeb(f)$ is orientation-preserving homeomorphic to $\Gamma$, each $N_i$ is a component of level set of $f$ which corresponds to the edge $e_i$ and each $W_i^\pm$ corresponds to $a_i^\pm$. Moreover, if $\dim W \geq 3$ and the maximum degree of a vertex in $\Gamma$ is not greater than $3$, then $f$ can be taken to be simple.
\end{theorem}

\begin{proof}
	Let $\Gamma'$ be a tree obtained from $\Gamma$ by cutting along all edges $e_i$ as it is presented in Figure \ref{figure:cutting_graph_along_edge}.
	Denote by $c^-_i$ and $c^+_i$ the vertices of $\Gamma'$ of outdegree $0$ and indegree $0$, respectively, obtained by cutting $\Gamma$ along the edge $e_i$.

	\begin{figure}[h]
		\centering
		\begin{tikzpicture}[scale=1]

			\filldraw (0,0) circle (1.7pt);
			\filldraw (0,2) circle (1.7pt);
			
			\draw (0,0)  to[out=60,in=300] (0,2);
			\draw[very thick, red] (0.1,1) -- (0.5,1);

			\draw[dotted, thick] (-0.15,-0.2) -- (-0.3,-0.4);
			\draw[dotted, thick] (-0.15,0.2) -- (-0.3,0.4);
			
			\draw[dotted, thick] (-0.15,1.8) -- (-0.3,1.6);
			\draw[dotted, thick] (-0.15,2.2) -- (-0.3,2.4);

			\draw (0.6,0.4) node {$e_i$};

			\draw [->, looseness=2, snake it ] (1.1,1) -- (1.9,1);

			\filldraw (3,0) circle (1.7pt);
			\filldraw (3,2) circle (1.7pt);
			\filldraw (3.28,0.75) circle (1.7pt) node[align=left, right] {$c_i^-$};
			\filldraw (3.28,1.25) circle (1.7pt) node[align=left, right] {$c_i^+$};
			
			\draw (3,0) to[bend right] (3.28,0.75);
			\draw (3,2) to[bend left] (3.28,1.25);

			\draw[dotted, thick] (2.85,-0.2) -- (2.7,-0.4);
			\draw[dotted, thick] (2.85,0.2) -- (2.7,0.4);
			
			\draw[dotted, thick] (2.85,1.8) -- (2.7,1.6);
			\draw[dotted, thick] (2.85,2.2) -- (2.7,2.4);

		\end{tikzpicture}
		\caption{Cutting along edge.}\label{figure:cutting_graph_along_edge}
	\end{figure}
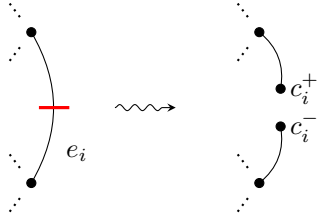

	For the case $\dim W=2$ consider the closed surface $Q$ which is formed from $W|\N$ by attaching discs to all its boundary components. By \cite[Theorem 5.4 and Remark 5.5]{Michalak-TMNA} there is a Morse function $g\colon Q \to \R$ whose Reeb graph is orientation-preserving homeomorphic~to~$\Gamma'$.

	For each vertex from $a^\pm_1,\ldots,a^\pm_{|\pi_0(W_\pm)|}$ and all $c^\pm_i$ consider a disc $D_l$ in $Q$ centre at the corresponding extremum of $g$ and whose boundary is a component of a level set of~$g$. Clearly, there is a diffeomorphism $h\colon W|\N \to Q\setminus \bigcup_l \int D_l$ and by homogeneity we may assume that it maps $W^\pm_i$ (circles $P_{\pm 1}(N_i)$) to boundaries of $D_l$ corresponding to $a^\pm_j$ (corresponding~to~$c^\pm_i$). Thus we define a Morse function $f\colon W|\N \to \R$ by $f=g\circ h$. Obviously, we can rescale $f$ to assume that its value on $P_{-1}(N_i)$ is smaller than its value on $P_{+1}(N_i)$ for each $i$. Thus we may extend $f$ on $P(\N)$ to a Morse function on $W$ whose Reeb graph, by the construction, is orientation-preserving homeomorphic to $\Gamma$, and $N_i$ corresponds to the edge $e_i$ for each $i$. Again, it can be rescaled to be a function on the triad $(W,W_-,W_+)$.

	Now, let $\dim W \geq 3$. We proceed analogously as in the proof of Theorem \ref{theorem:realization-DCG} \cite[Theorem 6.4]{Michalak-DCG} with the difference that the manifold has a boundary. However, we deal with the simplest case when graph is a tree. Steps 1 and 2 of the mentioned proof reduce the problem to the case when $\Gamma'$ has only vertices of degrees $1$ and $3$ and is primitive, i.e. there is no oriented (directed) path from a vertex with indegree $2$ to a vertex with outdegree $2$. 
	By Theorem~\ref{theorem:factorization_by_Reeb_epimorphism} for the empty system of hypersurfaces in $W|\N$ there is a simple Morse function $g\colon W|\N \to \R$ whose Reeb graph is the initial graph admissible for the triad $(W|\N,W_-\sqcup P_{1}(\N),W_+\sqcup P_{-1}(\N))$. 
	We may increase (or decrease if necessary) the number of vertices of degree $1$ by using the modifications (8) and (9), so that $\reeb(g)$ is the initial graph with the same numbers of vertices of indegree $0$ and outdegree $0$ as $\Gamma'$. Moreover, since $\Gamma'$ and $\reeb(g)$ are primitive trees, it forces the same numbers of vertices of indegree $2$ and of outdegree $2$. Thus it suffices to appropriately rearrange vertices of degree $3$ to produce $\Gamma'$ from $\reeb(g)$.

	For this purpose, we introduce the combinatorial modification number (13) presented in Figure \ref{figure:rearrangement_in_modification_(6)}, which allows us to transfer a vertex $v$ of indegree $2$ onto the second outgoing edge from a vertex $w$ of outdegee $2$ adjacent to $v$. The analogous modification for graphs with opposite orientations is numbered as (14). Since the graphs $\Gamma'$ and $\reeb(g)$ are primitive, small neighbourhoods of two adjacent vertices of degree $3$ look like in the modifications (4), (5), (13) or (14). Note that these modifications are two-sided, i.e. they work in both directions. We will show that $\Gamma'$ can be transformed to the initial graph by using them, and so $\reeb(g)$ can be transformed to $\Gamma'$ obtaining a simple Morse function $f'\colon W|\N \to \R$ whose Reeb graph is orientation-preserving homeomorphic to $\Gamma'$. Moreover, previously rearranging vertices in the initial graph $\reeb(g)$ using the modifications (4) and (5) we may ensure that the distinguished vertices of degree $1$ in the homeomorphic graphs $\Gamma'$ and $\reeb(f')$ correspond to appropriate components of $\partial (W|\N)$. Again, as in the case for surfaces, we can rescale $f'$ to assume that its value on $P_{-1}(N_i)$ is smaller than its value on $P_{+1}(N_i)$ for each $i$ and extend $f'$ on $P(\N)$ to a Morse function $f$ on $W$ whose Reeb graph, by the construction, is orientation-preserving homeomorphic to $\Gamma$, and which satisfies all desired conditions.

	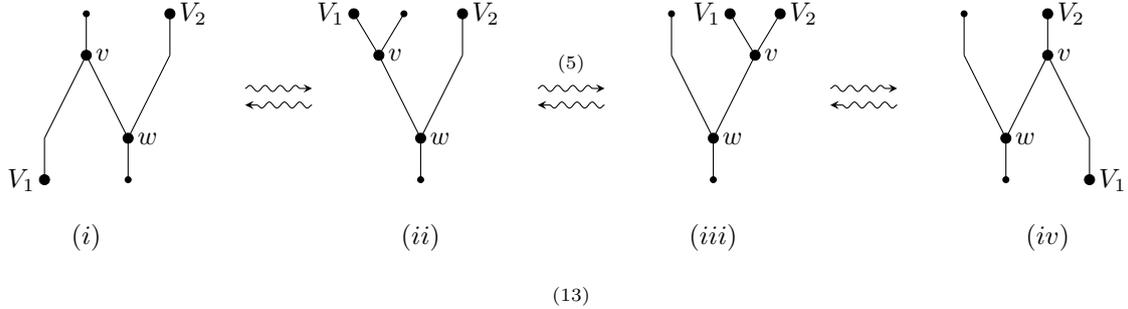
\begin{figure}[h]
		\centering
		
		\begin{tikzpicture}[scale=1.0]


			\draw (3,0) -- (3,0.5);
			\draw (2,0) -- (2,0.5);
			\draw (2,0.5) -- (2.5,1.5);
			\draw (3.5,1.5) -- (3,0.5);
			\draw (2.5,1.5) -- (3,0.5);
			\draw (3.5,1.5) -- (3.5,2);
			\draw (2.5,1.5) -- (2.5,2);

			\filldraw (2.5,1.5) circle (1.7pt)   node[align=left, right] {$v$};
			\filldraw (3,0.5) circle (1.7pt)   node[align=left, right] {$w$};
			
			\filldraw (3,0) circle (1pt);
			\filldraw (2,0) circle (1.7pt) node[align=right, left] {$V_1$};
			\filldraw (3.5,2) circle (1.7pt)  node[align=left, right] {$V_2$};
			\filldraw (2.5,2) circle (1pt);

			\draw [->, looseness=2, snake it ] (4.15,1.1) -- (4.95,1.1);
			\draw [->, looseness=2, snake it ] (4.95,0.9) -- (4.15,0.9);
			\draw (2.5,-0.7) node{($i$)};
			

			\draw (6,0) -- (6,0.5);
			\draw (6,0.5) -- (5.5,1.5);
			\draw (6.5,1.5) -- (6,0.5);
			\draw (5.5,1.5) -- (5.2,2);
			\draw (6.5,1.5) -- (6.5,2);
			\draw (5.5,1.5) -- (5.8,2);

			\filldraw (5.5,1.5) circle (1.7pt)   node[align=left, right] {$v$};
			\filldraw (6,0.5) circle (1.7pt)   node[align=left, right] {$w$};
			
			\filldraw (6,0) circle (1pt);
			\filldraw (5.2,2) circle (1.7pt)  node[align=right, left] {$V_1$};
			\filldraw (6.5,2) circle (1.7pt)  node[align=left, right] {$V_2$};
			\filldraw (5.8,2) circle (1pt);

			\draw [->, looseness=2, snake it ] (7.15,1.1) -- (7.95,1.1);
			\draw [->, looseness=2, snake it ] (7.95,0.9) -- (7.15,0.9);
			\draw (7.55,1.4) node{(5)};
			\draw (6,-0.7) node{($ii$)};
			

			\draw (9,0) -- (9,0.5);
			\draw (9,0.5) -- (8.5,1.5);
			\draw (9.5,1.5) -- (9,0.5);
			\draw (8.5,1.5) -- (8.5,2);
			\draw (9.5,1.5) -- (9.8,2);
			\draw (9.5,1.5) -- (9.2,2);

			\filldraw (9.5,1.5) circle (1.7pt)   node[align=left, right] {$v$};
			\filldraw (9,0.5) circle (1.7pt)   node[align=left, right] {$w$};
			
			\filldraw (9,0) circle (1pt);
			\filldraw (9.2,2) circle (1.7pt)  node[align=right, left] {$V_1$};
			\filldraw (9.8,2) circle (1.7pt)  node[align=left, right] {$V_2$};
			\filldraw (8.5,2) circle (1pt);

			\draw [->, looseness=2, snake it ] (10.15,1.1) -- (10.95,1.1);
			\draw [->, looseness=2, snake it ] (10.95,0.9) -- (10.15,0.9);
			\draw (9,-0.7) node{($iii$)};


			\draw (12,0) -- (12,0.5);
			\draw (12,0.5) -- (11.5,1.5);
			\draw (12.5,1.5) -- (12,0.5);
			\draw (11.5,1.5) -- (11.5,2);
			\draw (12.5,1.5) -- (12.5,2);
			\draw (12.5,1.5) -- (13,0.5);
			\draw (13,0) -- (13,0.5);

			\filldraw (12.5,1.5) circle (1.7pt)   node[align=left, right] {$v$};
			\filldraw (12,0.5) circle (1.7pt)   node[align=left, right] {$w$};
			
			\filldraw (12,0) circle (1pt);
			\filldraw (13,0) circle (1.7pt)  node[align=left, right] {$V_1$};
			\filldraw (12.5,2) circle (1.7pt)  node[align=left, right] {$V_2$};
			\filldraw (11.5,2) circle (1pt);

			\draw (12.5,-0.7) node{($iv$)};


			\draw (7.55,-1.4) node{(13)};

		\end{tikzpicture}
		\caption{The combinatorial modification number (13) of the Reeb graph of a simple Morse function. It transfers a vertex $v$ of degree $3$ and indegree $2$	onto the second outgoing edge from a vertex $w$ of degree $3$ and outdegree $2$. The situation with opposite directions of graphs leads to modification (14).}\label{figure:rearrangement_in_modification_(6)}
	\end{figure}

	Thus, take an oriented path $\tau$ in $\Gamma'$ between vertices of degree $1$ and with maximum number of vertices of degree $3$. Assume that there is a vertex $v$ of degree $3$ outside $\tau$, which is adjacent to a vertex $w$ of degree $3$ on $\tau$. Without loss of generality assume that $w$ has outdegree $2$. If $v$ also has outdegree $2$, then use the modification (5) to move $v$ on the path $\tau$. For the second case when $v$ has indegree $2$, first use (5) to move $w$ up along $\tau$ as long as it is possible. Then use the modification (13) for $v$ and $w$ to move $v$ on $\tau$. The obtained graph is still primitive and has increased number of vertices of degree $3$ on $\tau$. Repeating this procedure as long as there is a vertex of degree $3$ outside $\tau$ we obtain the initial graph.
	
	It only remains to provide the construction of the modification (13). Suppose $g\colon Q\to \R$ is a simple Morse function on a~triad $(Q,Q_-,Q_+)$ such that $\reeb(g)$ is isomorphic to the graph in Figure \ref{figure:rearrangement_in_modification_(6)} $(i)$. Let $v$ and $w$ be adjacent vertices of $\reeb(g)$ as in the figure and let $v_1$ be a vertex of degree $1$ in $\reeb(g)$ adjacent to $v$ and corresponding to the submanifold $V_1$ of $Q$, so $V_1$ is a single point (extremum of $g$) or component of $Q_-$. If it is an extremum, then use the modification (8) and then (9) to obtain the graph from Figure \ref{figure:rearrangement_in_modification_(6)} $(ii)$. Therefore assume $V_1$ is a component of $Q_-$. First, rescale the function along the edge between $v$ and $v_1$ so that the value on $v_1$ is greater than the value on $w$. Take a neighbourhood $U$ of this edge containing no other vertices than $v$ and $v_1$ such that the corresponding submanifold $S$ of $Q$ forms a triad $(S,V_1\sqcup S_1,S_2)$. Change a function $g$ on $S$ by defining a new simple and ordered Morse function on $(S,S_1,V_1\sqcup S_2)$ without critical points being extremum. By \cite[Proposition 3.2]{Michalak-DCG} this produces a function with its Reeb graph as in $(ii)$. The modification (5) leads to the case $(iii)$, and the analogous argument as before allows us to pass~to~$(iv)$. Finally, the modification (14) for a simple Morse function~$g$ can be obtained from (13) for the function $-g$.
\end{proof}

\begin{remark}\label{remark:simple_function_on_surface}
	Note that the constructed function $f$ in the proof of the above theorem in the case of surfaces can be simple if $W|\N$ is non-orientable or has genus $0$ since then the function $g$ in the proof can be taken to be simple (see \cite{Michalak-TMNA}). It is the case for example if $W$ is non-orientable of odd genus.
\end{remark}


\subsection{Realization of epimorphisms onto free groups as Reeb epimorphisms}

The previous theorem allows us to provide a
realization of an epimorphism $\pi_1(W) \to \pi_1(\Gamma)$ as the Reeb epimorphism of a Morse function on $(W,W_-,W_+)$. Note that by Theorem \ref{theorem:Reeb_epi_iff_regular_and_independent} we need to assume that it factorizes through $\pi_1(W)/\langle\pi_1(\partial W)\rangle^{\pi_1(W)}$.

\begin{theorem}\label{theorem:realization_of_Reeb_epi}
	Let $\Gamma$ be a finite connected graph with good orientation, 
	$(W,W_-,W_+)$ be a smooth triad and assume that $\varphi\colon \pi_1(W) \to \pi_1(\Gamma)$ is an epimorphism factorized through $\pi_1(W)/\langle\pi_1(\partial W)\rangle^{\pi_1(W)}$. Then there is a Morse function $f\colon W \to \R$ on $(W,W_-,W_+)$ such that $\reeb(f)$ is orientation-preserving homeomorphic to $\Gamma$ and under this identification the Reeb epimorphism of $f$ is equal to $\varphi$. Moreover, if $W$ is not a surface and the maximum degree of a vertex in $\Gamma$ is not greater than $3$, then $f$ can be simple.
\end{theorem}

\begin{proof}
	Take a spanning tree $T$ of $\Gamma$ and order the edges outside $T$ by $e_1,\ldots,e_r$. Take the quotient map $p_T\colon \Gamma \to \Gamma/T = \bigvee^r_{i=1} \es^1$ which maps $e_i$ onto $i$-th circle. By Theorem \ref{theorem:Reeb_epi_iff_regular_and_independent} the epimorphism $(p_T)_\# \circ \varphi$ is induced by an independent and regular system $\N=(N_1,\ldots,N_r)$ of hypersurfaces without boundary in $W$ of size $r$. By Theorem \ref{theorem:realization_of_graph_for_surfaces_with_prescribed_level_sets} there is a Morse function $f\colon W \to \R$ whose Reeb graph is $\Gamma$ up to vertices of degree $2$ and $N_i$ corresponds to $e_i$. If $q\colon W \to \reeb(f)$ is the quotient map, then by the construction $(p_T)_\# \circ q_\# = \varphi_\N= (p_T)_\# \circ \varphi$. Since $(p_T)_\#$ is an isomorphism, $\varphi=q_\#$ is the Reeb epimorphism~of~$f$.
\end{proof}

\begin{corollary}\label{corollary:Reeb_epi_for_closed_manifolds}
	Let ${\varphi\colon \pi_1(M) \to \pi_1(\Gamma)}$ be an epimorphism, where $M$ is a closed manifold and $\Gamma$ is a finite connected graph with good orientation. Then there is a Morse function $f\colon M \to \R$ such that $\reeb(f)$ is orientation-preserving homeomorphic to $\Gamma$ and under this identification the Reeb epimorphism of $f$ is equal to $\varphi$. \qed 
\end{corollary}

\begin{remark}\label{remark:rigorous_representation-id_only_automorphism}
	The Reeb epimorphism of $f$ does not represent a unique epimorphism $\pi_1(W)\to \pi_1(\Gamma)$ in general, because it depends on the homeomorphism between $\Gamma$ and $\reeb(f)$. It is unique for oriented graphs such that the identity map is the only orientation-preserving automorphism. Otherwise, Theorem \ref{theorem:realization_of_graph_for_surfaces_with_prescribed_level_sets} provides more rigorous representation of a Reeb epimorphism if some additional data is given. For instance, assume that there are distinguished edges $e_1,\ldots,e_r$ outside a spanning tree $T$ of $\Gamma$ and a regular and independent system $\N$ of hypersurfaces inducing $(p_T)_\# \circ \varphi$. Then the condition that each $N_i$ corresponds to $e_i$ implies the uniqueness of an epimorphism $\pi_1(W)\to \pi_1(\Gamma)$ represented by the Reeb epimorphism of $f$.
\end{remark}

\begin{remark}\label{remark:Saeki}
	Independently, O. Saeki \cite{Saeki_Reeb_spaces} has proven a similar result for a closed manifold $M$, which provides a representation of an epimorphism $\varphi \colon \pi_1(M) \to \pi_1(\Gamma)$, for any finite graph $\Gamma$ without loops, as the Reeb epimorphism of a smooth function with finitely many critical values. In addition, since the constructed function has degenerate critical points, it can realize $\Gamma$ as the Reeb graph up to isomorphism of graphs. However, the number of vertices of degree $2$ in the Reeb graph of a Morse function cannot be arbitrary (see, for instance \cite[Theorem 5.6]{Michalak-TMNA}), thus we need to ignore them in our construction of a Morse function. Moreover, the above theorem together with Theorem \ref{theorem:realization_of_graph_for_surfaces_with_prescribed_level_sets} provide more rigorous representation of $\varphi$ as the Reeb epimorphism, since we control the components of level sets of the function. It can be crucial in applications of Reeb epimorphisms, e.g. in topological conjugacy of Morse functions (see Section \ref{section:conjugacy_of_Morse_functions}). Finally, our result deals also with manifolds with boundary.
\end{remark}

We showed that for manifolds of dimension at least $3$ any epimorphism $\pi_1(M)\to \pi_1(\Gamma)$ is represented as the Reeb epimorphism of a simple Morse function provided that $\Gamma$ satisfies necessary conditions: it has a good orientation and the maximum degree of its vertices is not greater than $3$. Now, let us investigate when in Theorem \ref{theorem:realization_of_Reeb_epi} one can take a simple Morse function in the case of surfaces. 

\begin{lemma}\label{lemma:simple_iff_maximal_cycle_rank}
	Let $\Sigma$ be a closed surface of Euler characteristic $\chi(\Sigma) = 2-2g$, so it is orientable of genus $g$ or non-orientable of genus $2g$, and let $f\colon \Sigma \to \R$ be a Morse function such that the maximum degree of a vertex in $\reeb(f)$ is not greater than $3$.
	Then $\beta_1(\reeb(f))=g$ if and only if $f$ is $\reeb$-simple and $\reeb(f)$ has no vertices of degree~$2$.
\end{lemma}

\begin{proof}
	Let $V$ and $E$ be the sets of vertices and edges of $\reeb(f)$, respectively. Then $\beta_1(\reeb(f)) = |E|-|V|+1$. Moreover, if $k_i$ is the number of critical points of $f$ of index $i$, then $\chi(\Sigma)=k_0-k_1+k_2$ and the number of vertices of degree $1$ in $\reeb(f)$ is equal to $k_0+k_2$. 
	Let $\Delta_2$ and $\Delta_3$ be the numbers of vertices of degree $2$ and $3$, respectively. Note that $2|E| = \sum_{v\in V} \deg(v) = k_0+k_2+2\Delta_2+3\Delta_3$ by Euler's handshaking lemma.
	Combining these equalities we obtain
	$$
	2\left(\beta_1(\reeb(f))-g\right) = 2|E|-2|V| + \chi(\Sigma) =  \Delta_3-k_1.
	$$
	Thus $\beta_1(\reeb(f))=g$ if and only if $k_1=\Delta_3$. However, $k_1 \geq \Delta_2 + \Delta_3$, so $k_1=\Delta_3$ if and only if $\Delta_2=0$ and each vertex of $\reeb(f)$ corresponds to a single critical point of $f$, so $f$ is $\reeb$-simple.	
\end{proof}

Note that a simple Morse function on a non-orientable surface of odd genus has always a vertex of degree $2$ in its Reeb graph.

We know that simple Morse functions on a closed orientable surface of genus $g$ have a unique property that their Reeb graphs have cycle rank equal to $g$. For a non-orientable surface $\Sigma$ one can construct simple Morse functions with Reeb graphs having arbitrary cycle rank between $0$ and $\reeb(\Sigma)$ (cf. Theorem \ref{thm:equivalent_conditions}). However, it turns out that simple Morse functions have also some unique property for non-orientable surfaces of even genus, which can be described in terms of Reeb epimorphisms.

\begin{proposition}\label{proposition:unique_strong_equi_class_for_simple_Morse_functions}
	Let $\Gamma$ be a finite connected graph with good orientation such that $\beta_1(\Gamma) < g$. Then there is a unique strong equivalence class $\Xi$ of epimorphisms $\pi_1(S_{2g})\to \pi_1(\Gamma)$ such that for any simple Morse function $f\colon S_{2g}\to \R$  having $\reeb(f) = \Gamma$ its Reeb epimorphism belongs to $\Xi$.	
\end{proposition}

\begin{proof}
	Let $r:= \beta_1(\reeb(f))$ and $\N=(N_1,\ldots,N_r)$ be an independent and regular system of hypersurfaces in $S_{2g}$ which are connected components of level sets of $f$ and which correspond to edges outside some spanning tree of $\reeb(f)$.  We claim that $S_{2g}|\N$ is non-orientable. Thus assume that $S_{2g}|\N$ is orientable. First, note that $f|_{S_{2g}|\N}$ is a simple Morse function and its Reeb graph $\reeb(f|_{S_{2g}|\N})$ is a tree. Therefore $S_{2g}|\N$ has a genus $0$ as a surface with boundary, i.e. it is a sphere with discs removed. This implies that $r=g$, a~contradiction.
	
	Therefore as in the proof of Theorem \ref{theorem:calculations_of_cobordisms_for_surfaces} any two Reeb epimorphisms of simple Morse functions $S_{2g}\to \R$ with $\Gamma$ as Reeb graphs are strongly equivalent and $\Xi$ is represented by a system of hypersurfaces whose complement is non-orientable.
\end{proof}

Note that in fact the above proposition is also true for any non-orientable surface $S_{2g+1}$ of odd genus since by Theorem \ref{thm:grigorchuk} there is only one strong equivalence class of epimorphisms $\pi_1(S_{2g+1})\to \pi_1(\Gamma)$.

The following corollary follows by Remark \ref{remark:simple_function_on_surface}, Theorem \ref{theorem:realization_of_Reeb_epi}, Lemma \ref{lemma:simple_iff_maximal_cycle_rank} and Proposition \ref{proposition:unique_strong_equi_class_for_simple_Morse_functions}.

\begin{corollary}\label{corollary:Reeb_epi_for_simple_Morse}
	Let $\Gamma$ be a finite connected graph with good orientation such that the maximum degree of its vertex is not greater than $3$, let $\Sigma$ be a closed surface and $\mathcal{R}_{epi}(\Sigma)$ be the set of all Reeb epimorphisms of simple Morse functions on $\Sigma$. Take an epimorphism $\psi\colon \pi_1(\Sigma) \to \pi_1(\Gamma)$.
	\begin{itemize}
		\item If $\Sigma$ is orientable of genus $g$, then $\psi \in \reeb_{epi}(\Sigma)$ if and only if $\beta_1(\Gamma)=g$.
		\item If $\Sigma$ is non-orientable of odd genus, then any $\psi \in \reeb_{epi}(\Sigma)$.
		\item If  $\Sigma$ is non-orientable of even genus $2g$, then $\psi \in \reeb_{epi}(\Sigma)$ if and only if $\beta_1(\Gamma)=g$, or $\beta_1(\Gamma)<g$ and $\psi$ belongs to a unique strong equivalence class $\Xi$ of epimorphisms $\pi_1(\Sigma) \to \pi_1(\Gamma)$ represented by systems of hypersurfaces whose complements are non-orientable.
	\end{itemize}
	\qed
\end{corollary}

\subsection{Extendability of independent systems of hypersurfaces}\label{section:extendability_of_systems}

Let $\mathcal{N}=(N_1,\ldots,N_r)$ be an independent and regular system of hypersurfaces in $W$. We say that $\mathcal{N}$ is \emph{extended} by a system $\mathcal{N'}$ if $\N'$ is also a regular and independent system such that $\N\subset \N'$ and their framings determine the same orientation of the normal bundle of $\N$ in $W$.

\begin{proposition} \label{proposition:size_of_maximum_extension_of_system}
	Let $\mathcal{N}$ be an independent and regular system without boundary in $W$ of size $r$. Then
	$$
	\corank\left(\pi_1(W)/\geng{\pi_1(\N)}^{\pi_1(W)}\right) = \corank\left(\pi_1(W|\N)/\geng{\pi_1(\partial P(\N))}^{\pi_1(W|\N)}\right) +r
	$$
	and it is the maximum size of  an independent and regular system
	without boundary in $W$ which extends $\mathcal{N}$. In particular,
	for a closed manifold $M$ we get
	$$
	\reeb(M|\N) = \corank\left(\pi_1(M)/\geng{\pi_1(\N)}^{\pi_1(M)}\right) -r.
	$$
\end{proposition}

\begin{proof}
	Suppose we have a $2$-sided connected submanifold $N=N_1$ without boundary with product neighbourhood $P(N)$ in a compact manifold $W$ such that $W|N$ is connected. Thus $W$ is obtained from $W|N$ be gluing the components of boundary $\partial(W|N) = P_{-1}(N) \sqcup P_{1}(N) \sqcup \partial W$ using a diffeomorphism $h\colon P_{-1}(N) \to P_{1}(N)$. It is known (see \cite[Chapter~IV]{Lyndon-Schupp}) that $\pi_1(W)$ is the HNN extension of $\pi_1(W|N)$ relative to $h_\# \colon H_{-1} \to H_{1}$, where $H_{t} = \pi_1(P_{t}(N)) < \pi_1(W|N)$. In other words, $\pi_1(W)$ is the free product $\pi_1(W|N) * \Z$ divided by the normal closure of $\{ u \omega u^{-1} h_\#(\omega)^{-1}\, \colon\, \omega \in H_{-1}\}$, where $u$ is the stable letter which generates $\Z$. The group $\pi_1(W|N)$ is a subgroup of $\pi_1(W)$ and the groups $H_{-1}$ and $H_1$ are conjugate in $\pi_1(W)$. In fact, the normal subgroup $\pi_1(N)^{\pi_1(W)}$ in $\pi_1(W)$ is equal to $\geng{H_{-1}}^{\pi_1(W)} = \geng{H_{1}}^{\pi_1(W)} = \geng{H_{-1},H_1}^{\pi_1(W|N)}$. Therefore $\pi_1(W)/\pi_1(N)^{\pi_1(W)}$ is isomorphic to
	$$
	\pi_1(W)/\geng{H_{-1},H_1}^{\pi_1(W|N)} \cong \pi_1(W|N)/\geng{H_{-1},H_1}^{\pi_1(W|N)} * \Z.
	$$
	It gives the first part of the proposition for $r=1$ since $\corank(G*H)=\corank(G)+\corank(H)$ (see \cite{Cornea}). The general case follows by considering all submanifolds $N_i$ simultaneously and HNN extension with $r$ stable letters.

	The description of the number on both sides of this equality follows by Proposition~\ref{proposition:cones-systems_omitting_part_of_boundary}.
\end{proof}

\begin{example}\label{example:product_surface_circle_not_ME}
	The Reeb number of a compact manifold $W$ with non-empty boundary can be smaller than $C(W)$. For example, let $M= \Sigma \times \es^1$,
	where $\Sigma$ is a closed surface of the Euler characteristic $\chi(\Sigma) = 2 - k \leq 0$. Then 
	$$
	\reeb(M) = \corank(\pi_1(\Sigma)\times \Z) =  \max\left(\left\lfloor \frac{k}{2} \right\rfloor,1\right) = \left\lfloor \frac{k}{2} \right\rfloor \geq 1
	$$
	(see \cite{Gelbukh:corank, Michalak-TMNA}).
	Let  $\N=(\Sigma \times \{1\})$ and $W:= M |\N \cong \Sigma \times [0,1]$. Then $C(W) = \corank(\pi_1(\Sigma)) =  \left\lfloor \frac{k}{2} \right\rfloor$. However,
	$\reeb(W) = \corank\left((\pi_1(\Sigma) \times \Z) / \pi_1(\Sigma)\right) -1 = 0$ by
	Proposition \ref{proposition:size_of_maximum_extension_of_system}. Therefore $\N$ cannot be extended to a regular and independent system of hypersurfaces in $M$ of a larger size.
\end{example}

\begin{example}
	Let $\Sigma_{g,h}$ and $S_{g,h}$ denote, respectively, an orientable and non-orientable surface of genus $g$ with $h\geq 1$ open discs removed. Then 
	\begin{itemize}
		\item $\reeb(\Sigma_{g,h}) = g$ and $C(\Sigma_{g,h})=2g+h-1$,
		\item  $\reeb(S_{g,h}) = \left\lfloor \frac{g}{2} \right\rfloor$ and $C(S_{g,h})=g+h-1$.
	\end{itemize}
	Indeed, by Proposition \ref{proposition:cones-systems_omitting_part_of_boundary} we have $\reeb(W) = \corank(\pi_1(\cone_{\partial W}(W)))$. It gives the calculation of Reeb numbers since
	$\cone_{\partial \Sigma_{g,h}}(\Sigma_{g,h}) \cong \Sigma_g$ and $\cone_{\partial S_{g,h}}(S_{g,h}) \cong S_g$. The calculation of $C(W)$ follows by Theorem \ref{theorem:Jaco for smooth} and from the
	fact that $\pi_1(\Sigma_{g,h})=F_{2g+h-1}$ and $\pi_1(S_{g,h}) = F_{g+h-1}$ .
\end{example}

\begin{corollary}\label{corollary:extendability_of_systems_in_surfaces}
	Any independent, regular and without boundary system $\mathcal{N}$ of hypersurfaces in a compact surface $\Sigma$ can be extended to such a
	system of size $\reeb(\Sigma)$.
\end{corollary}

\begin{proof}
	Let $r$ be the size of $\mathcal{N}$. Since $\Sigma$ is two-dimensional, $\N$ consists of circles in $\Sigma$. It is easily seen by the classification theorem of compact surfaces that if $\Sigma=\Sigma_{g,h}$ then  $\Sigma|\N \cong \Sigma_{g-r,h+2r}$, and if $\Sigma=S_{g,h}$ then $\Sigma|N \cong S_{g-2r,h+2r}$ or $\Sigma|N \cong \Sigma_{g/2-r,h+2r}$ (the latter case can only occur if $g$ is even). By the above example in all the cases $\reeb(\Sigma|\N) = \reeb(\Sigma) - r$, so $\mathcal{N}$ can be extended to the size $\reeb(\Sigma)$.
\end{proof}


\subsection{Topological conjugacy of Morse functions}\label{section:conjugacy_of_Morse_functions}

Now, we focus on relations between Reeb epimorphisms and Morse functions. The main trouble is that in general different Reeb epimorphisms have different codomains. Although fundamental groups of Reeb graphs with the same cycle ranks are isomorphic, they are not isomorphic in the canonical way. However, this ambiguity can be omitted for oriented graphs for which the identity map is the only orientation-preserving automorphisms (see Remark \ref{remark:rigorous_representation-id_only_automorphism}).

Let us restrict our attention to the case of a closed manifold $M$. The functions $f_1$ and $f_2$ on $M$ are called \emph{topologically conjugate} if there are a self-homeomorphism $h\colon M \to M$ and an orientation-preserving homeomorphism $\eta\colon \R \to \R$ such that $f_1 = \eta \circ f_2 \circ h$. In this case $h$ induces the unique homeomorphism $\overline{h}\colon \reeb(f_1) \to \reeb(f_2)$ such
that $\overline{h}\circ q_1 = q_2 \circ h$ and $\overline{f_1} = \eta \circ \overline{f_2}\circ \overline{h}$.

\begin{lemma}\label{lemma:equivalence by conjugacy}
	If $f_1$ and $f_2$ are simple Morse functions topologically conjugate by $h$, then their Reeb graphs are isomorphic through $\overline{h}$.
\end{lemma}

\begin{proof}
	If could be a vertex with degree $2$  in $\reeb(f_1)$ mapped by $\overline{h}$ to a point on an edge in $\reeb(f_2)$, then some smooth product triad would be mapped by $h^{-1}$ homeomorphically onto a smooth triad with exactly one non-degenerate critical point. It is a contradiction by the comparison of Euler characteristics.
\end{proof}

\begin{figure}[h]
	$$
	\begin{xy}
		(-6,18)*+{M}="M";
		(-6,0)*+{M}="M2";
		(15,18)*+{\reeb(f_1)}="R(f1)";
		(15,0)*+{\reeb(f_2)}="R(f2)";
		(36,0)*+{\R}="R2";
		(36,18)*+{\R}="R1";
		{\ar@{->}^{q_1} "M"; "R(f1)"};%
		{\ar@{->}_{h} "M"; "M2"};%
		{\ar@{->}^{q_2} "M2"; "R(f2)"};
		{\ar@{-->}_{\overline{h}} "R(f1)"; "R(f2)"};
		{\ar@{->}^{\overline{f_1}} "R(f1)"; "R1"};
		{\ar@{->}^{\overline{f_2}} "R(f2)"; "R2"};
		{\ar@{->}_{\eta} "R1"; "R2"};
	\end{xy}
	$$
	
	\caption{Topologically conjugate simple Morse functions have isomorphic Reeb graphs.}\label{figure:topological_equivalence_of_Morse_functions}
\end{figure}

We cannot say directly about equivalence and strong equivalence of Reeb epimorphisms since they have distinct codomains, even if Reeb graphs are isomorphic. By the above diagram we see that if $f_1$ and $ f_2$ are topologically conjugate by $h$, then $\overline{h}_\# \circ (q_1)_\#$ and $(q_2)_\#$ are strongly equivalent. Thus we say that two Reeb epimorphisms $\varphi_i\colon \pi_1(M) \to \pi_1(\reeb(f_i))$ are \emph{strongly equivalent} if they are strongly equivalent with respect to an isomorphism $k\colon \reeb(f_1)\to \reeb(f_2)$, i.e. $k_\# \circ \varphi_1$ and $\varphi_2$ are strongly equivalent. In general, it may depend on the chosen isomorphism.

The following theorem is a classical result in the theory of Morse functions space.

\begin{theorem}[{Kulinich \cite{Kulinich}, Sharko \cite{Sharko-functions}, cf. \cite{Fabio-Landi}}]\label{theorem:conjugacy=isomorphic_Reeb_graphs_surfaces}
	Two simple Morse functions on a closed orientable surface $\Sigma$ are topologically conjugate by $h\colon \Sigma \to \Sigma$ if and only if their Reeb graphs are isomorphic as oriented graphs through~$\overline{h}$.
\end{theorem}

This theorem allows us to give another proof of a part of Theorem \ref{thm:grigorchuk} for orientable surfaces which uses Reeb graphs. First, for any two epimorphisms $\pi_1(\Sigma_g) \to F_r$ we need to take systems which induce them. Then we extend them to systems of maximum size $\reeb(\Sigma_g) = g$ by Corollary \ref{corollary:extendability_of_systems_in_surfaces} and now we can represent induced epimorphisms $\pi_1(\Sigma_g) \to F_g$ by Reeb epimorphisms of simple Morse function whose Reeb graphs are the initial graphs (they have no vertices of degree $2$). Thus by Theorem \ref{theorem:conjugacy=isomorphic_Reeb_graphs_surfaces} the isomorphism of the Reeb graphs  is induced by a~self-homeomorphism of $\Sigma_g$ that maps one system to the another and gives a strong equivalence between considered epimorphisms.

\begin{remark}
	Theorem \ref{thm:grigorchuk} for non-orientable surfaces of even genus shows that the analogue of Theorem \ref{theorem:conjugacy=isomorphic_Reeb_graphs_surfaces} does not hold for them in general. In fact, we may construct two simple Morse functions on $S_{2g}$ whose Reeb graphs are~isomorphic, but their Reeb epimorphisms are not strongly equivalent.
	Thus we must endow Reeb graphs in additional information.
\end{remark}

Lychak--Prishlyak in their work \cite{Prishlyak_nonoriented} equipped Reeb graphs of simple Morse functions on non-orientable surfaces with signs $+$ or $-$ near vertices of degree $3$, which come from the compatibility of orientations during attaching handles in corresponding critical levels. To be precise, each sign is assigned to a pair of incident edges at a vertex $v$ of degree $3$, one of which is incoming to $v$ and the second one is outgoing from $v$.  For the procedure of the assignment of signs we refer the reader to \cite{Prishlyak_nonoriented}. Two \emph{Reeb graphs with signs} are called \emph{equivalent} if they are isomorphic and it is possible to obtain identical signs by the following operation: for a given edge reverse all signs assigned to it.

\begin{theorem}[Lychak--Prishlyak \cite{Prishlyak_nonoriented}]\label{thm:Prishlyak_conj}
	Two simple Morse functions on a closed nonorientable surface are topologically conjugate if and only if their Reeb graphs with signs are equivalent.
\end{theorem}

\begin{lemma}\label{lemma:configurations_of_signs}
	Let $\Gamma$ be a graph with good orientation whose vertices have degrees $1$ or $3$. Then there are exactly $2^r$ equivalence classes of graphs with signs, where $r=\beta_1(\Gamma)$.
\end{lemma}
\begin{proof}
	First, look at the case of the canonical graph, i.e. a graph presented in Figure \ref{figure:graphs_with_signs} (a).
	It is an easy exercise that any such graph with signs is equivalent to a configuration of the form showed in the figure,	where in the $r$ places of "?" we can put arbitrary signs. Moreover, all such $2^r$ configurations are non-equivalent. The same can be shown for the initial graph with configurations of signs as in Figure \ref{figure:graphs_with_signs} (b).
	
	Now, note that by \cite{Fabio-Landi} and \cite{Michalak-DCG}
	there is a sequence of modifications of Reeb graphs which transform $\Gamma$ to the canonical graph. It is left to the reader to check that these modifications for graphs with vertices of degree~$1$~or~$3$ do not change the number of non-equivalent configurations of signs.
\end{proof}

\begin{figure}[h]
	\centering
	\begin{tikzpicture}[scale=0.8]

		
		\filldraw (-6,0) circle (1.7pt);
		\filldraw (-6,0.7) circle (1.7pt)node[left,xshift=-0.1em, yshift=0.15em] {\footnotesize $+$} node[right,xshift=0.1em, yshift=0.15em] {\footnotesize $+$};
		\filldraw (-6,1.7) circle (1.7pt)node[left,xshift=-0.1em, yshift=0.15em] {\footnotesize $+$} node[right,xshift=0.1em, yshift=0.15em] {\footnotesize $?$};
		\filldraw (-6,3) circle (1.7pt)node[left,xshift=-0.1em, yshift=0.15em] {\footnotesize $+$} node[right,xshift=0.1em, yshift=0.15em] {\footnotesize $+$};
		\filldraw (-6,4) circle (1.7pt)node[left,xshift=-0.1em, yshift=0.15em] {\footnotesize $+$} node[right,xshift=0.1em, yshift=0.15em] {\footnotesize $?$};
		\filldraw (-6,4.7) circle (1.7pt)node[left,xshift=-0.1em, yshift=0.15em] {\footnotesize $+$} node[right,xshift=0.1em, yshift=0.15em] {\footnotesize $+$};
		\filldraw (-6,5.7) circle (1.7pt)node[left,xshift=-0.1em, yshift=0.15em] {\footnotesize $+$} node[right,xshift=0.1em, yshift=0.15em] {\footnotesize $?$};
		\filldraw (-6,6.4) circle (1.7pt);
		
		\draw (-6,0) -- (-6,0.7);
		
		\draw (-6,1.7) to[out=200,in=160] (-6,0.7);
		\draw (-6,1.7) to[out=340,in=20] (-6,0.7);
		\draw (-6,2) -- (-6,1.7);
		\draw (-6,2.7) -- (-6,3);
		\draw (-6,4) to[out=200,in=160] (-6,3);
		\draw (-6,4) to[out=340,in=20] (-6,3);
		\draw (-6,4.7) -- (-6,4);
		\draw (-6,5.7) to[out=200,in=160] (-6,4.7);
		\draw (-6,5.7) to[out=340,in=20] (-6,4.7);
		\draw (-6,6.4) -- (-6,5.7);
		
		\draw (-6,2.2) node { .};
		\draw (-6,2.35) node { .};
		\draw (-6,2.5) node { .};

		\draw (-6,-1) node {(a)};

		
		\filldraw (0,0.25) circle (1.7pt);
		\filldraw (0,1) circle (1.7pt) node[left,xshift=-0.1em, yshift=0.2em] {\footnotesize $+$} node[right,xshift=0.1em, yshift=0.2em] {\footnotesize $+$};
		\filldraw (0.5,1.5) circle (1.7pt) node[left,xshift=-0.1em, yshift=0.2em] {\footnotesize $+$} node[right,xshift=0.1em, yshift=0.2em] {\footnotesize $+$};
		\filldraw (1,2) circle (1.7pt)  node[left,xshift=-0.1em, yshift=0.2em] {\footnotesize $+$} node[right,xshift=0.1em, yshift=0.2em] {\footnotesize $+$};
		\filldraw (1.5,2.5) circle (1.7pt)  node[left,xshift=-0.1em, yshift=0.2em] {\footnotesize $+$} node[right,xshift=0.1em, yshift=0.2em] {\footnotesize $+$};
		\filldraw (2,3) circle (1.7pt)  node[left,xshift=-0.1em, yshift=0.15em] {\footnotesize $+$} node[right,xshift=0.1em, yshift=0.15em] {\footnotesize $+$};
		\filldraw (2,4) circle (1.7pt) node[left,xshift=-0.1em,yshift=-0.15em] {\footnotesize $+$} node[right,xshift=0.2em, yshift=-0.15em] {\footnotesize $?$};
		\filldraw (1.5,4.5) circle (1.7pt) node[left,xshift=-0.1em,yshift=-0.2em] {\footnotesize $+$} node[right,xshift=0.2em, yshift=-0.2em] {\footnotesize $?$};
		\filldraw (1,5) circle (1.7pt) node[left,xshift=-0.1em,yshift=-0.2em] {\footnotesize $+$} node[right,xshift=0.2em, yshift=-0.2em] {\footnotesize $?$};
		\filldraw (0.5,5.5) circle (1.7pt) node[left,xshift=-0.1em,yshift=-0.2em] {\footnotesize $+$} node[right,xshift=0.2em, yshift=-0.2em] {\footnotesize $?$};
		\filldraw (0,6) circle (1.7pt) node[left,xshift=-0.1em,yshift=-0.2em] {\footnotesize $+$} node[right,xshift=0.2em, yshift=-0.2em] {\footnotesize $?$};
		\filldraw (0,6.75) circle (1.7pt);
		
		\draw (0,0.25) -- (0,1);
		\draw (0.5,1.5) -- (0,1);
		\draw (0.5,1.5) -- (1,2);
		\draw (1.1,2.1) -- (1,2);
		\draw (1.4,2.4) -- (1.5,2.5);
		\draw (1.5,2.5) -- (2,3);

		\draw (0,6.75) -- (0,6);
		\draw (0.5,5.5) -- (0,6);
		\draw (0.5,5.5) -- (1,5);
		\draw (1.1,4.9) -- (1,5);
		\draw (1.4,4.6) -- (1.5,4.5);
		\draw (1.5,4.5) -- (2,4);

		\draw (2,4) to[out=200,in=160] (2,3);
		\draw (2,4) to[out=340,in=20] (2,3);

		\draw (1.5,4.5) to[out=225,in=135] (1.5,2.5);
		\draw (1,5) to[out=225,in=135] (1,2);
		\draw (0.5,5.5) to[out=225,in=135] (0.5,1.5);
		\draw (0,6) to[out=225,in=135] (0,1);

		\draw (1.34,2.34) node { .};
		\draw (1.25,2.25) node { .};
		\draw (1.16,2.16) node { .};

		\draw (1.34,4.66) node { .};
		\draw (1.25,4.75) node { .};
		\draw (1.16,4.84) node { .};

		\draw (0,-1) node {(b)};

	\end{tikzpicture}
	
	\caption{The canonical and initial graphs with signs.}\label{figure:graphs_with_signs}
\end{figure}
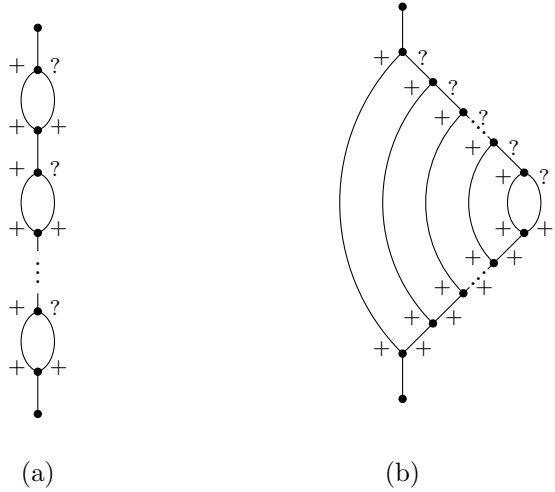

\begin{remark}
	For a graph with vertices of degree $2$ the number of non-equivalent configurations of signs may vary depending on the position of these vertices in the graph. Moreover, note that the Reeb graph of a simple Morse function on a non-orientable surface of odd genus has always a vertex of degree $2$.
	The same is true if a surface is non-orientable of even genus $2g$ and the Reeb graph has a cycle rank smaller than $g$. In the view of Lemma \ref{lemma:simple_iff_maximal_cycle_rank} it is reasonable to consider the case of simple Morse functions on a non-orientable surface of genus $2g$ whose Reeb graphs have cycle rank equal to $g$.
\end{remark}

\begin{theorem}\label{thm:conjugacy=strong equiv_Reeb_epi}
	Let $f_1, f_2\colon S_{2g} \to \R$ be simple Morse functions on a closed non-orientable surface of genus $2g$ such that $\beta_1(\reeb(f_1))=g=\beta_1(\reeb(f_2))$. Then they are topologically conjugate 
	if and only if their Reeb graphs are isomorphic 
	and their Reeb epimorphisms are strongly equivalent.
\end{theorem}

\begin{proof}
	Denote by $\mathcal{M}(M,\Gamma)$ the set of all simple Morse functions $f$ on $M$ whose Reeb graphs are isomorphic to a graph $\Gamma$ and by $\mathcal{M}(M,\Gamma)/_{\operatorname{t. c.}}$ the set of their topological conjugacy classes. Moreover, let $\operatorname{Signs}(\Gamma)$ be the set of equivalence classes of configurations of signs in $\Gamma$. By Theorem \ref{thm:Prishlyak_conj} a natural map $\mathcal{M}(S_{2g},\Gamma)/_{\operatorname{t. c.}} \to \operatorname{Signs}(\Gamma)$ associating with a function its configuration of signs in its Reeb graphs as in \cite{Prishlyak_nonoriented} is injective. Now, let $\Gamma$ has vertices of degrees $1$ and $3$ and $\beta_1(\Gamma)=g$. Then for any configuration of signs in $\Gamma$ except the one with only pluses we can produce a simple Morse function on $S_{2g}$ which realizes it (see \cite{Michalak-TMNA} for a procedure). The configuration of signs with only pluses leads to a function on an orientable surface. By Lemma \ref{lemma:configurations_of_signs} the set $\operatorname{Signs}(\Gamma)$ has $2^r$ elements, so $\mathcal{M}(S_{2g},\Gamma)/_{\operatorname{t. c.}}$ has $2^r-1$ elements.
	
	Now, by Theorem \ref{theorem:realization_of_Reeb_epi} and Corollary \ref{corollary:Reeb_epi_for_simple_Morse} the map from $\mathcal{M}(S_{2g},\Gamma)/_{\operatorname{t. c.}}$ to the set of strong equivalence classes of Reeb epimorphisms of functions from $\mathcal{M}(S_{2g},\Gamma)$ is surjective. Since the latter set has also $2^r-1$ elements by Theorem \ref{thm:grigorchuk}, it is a bijection and the theorem is proved.	
\end{proof}


\section{Final remarks}\label{section:prospects}

The presented approach could be more used in analysis to study $C^{1}$-function with finitely many critical points, since it assigns such a function  a combinatoric invariant (the Reeb graph), or respectively a non-commutative algebraic invariant (the Reeb epimorphism). Conversely, it gives a geometric description of the set of homomorphisms from a fundamental group of a compact manifold into a finitely generated free group. At now the following problems seem to be next natural steps in such studies.

\vskip 8pt
	1. One of the problems we will focus on is the question about extendability of independent systems of hypersurfaces (cf. Section \ref{section:extendability_of_systems}). For an independent and regular system $\N$ let us denote by $E(\N)$ the maximum size of an independent and regular system which extends $\N$ and by $F(\varphi_\N)$ the maximum rank of a free group onto which there is an epimorphism which factorizes~$\varphi_\N$. It is clear that $E(\N) \leq F(\varphi_\N)$. We would like to know for which closed manifolds $M$ it is the equality for any independent system of hypersurfaces in $M$.
	Using Theorem \ref{theorem:epimorphism_is_induced_by_regular_and_independent_system} it can be shown that $\N$ is framed cobordant to a system which can be extended to size $F(\varphi_\N)$, however we do not know whether $\N$ can be extended itself. In particular, we will investigate the case when all numbers $F(\varphi_\N)$ are equal to $\corank(\pi_1(M))$.
	This problem can be also seen more algebraically. By Proposition \ref{proposition:size_of_maximum_extension_of_system} the extendability of $\N$ to the size $\corank(\pi_1(M))$ is equivalent to the equality $\corank(\pi_1(M)) = \corank\left(\pi_1(M)/\geng{\pi_1(\N)}^{\pi_1(M)}\right).$ Moreover, since $E(\N) = \reeb(M|N)+r$, where $r$ is the size of $\N$, this problem is related to the computability of the corank and Reeb number. 
	
\vskip 8pt
	2.The last section shows that the relations between conjugacy of simple Morse functions and strong equivalence of Reeb epimorphisms are quite complicated even for surfaces. The careful investigation of these relations should be done for higher-dimensional manifolds, especially for $3$-manifolds. However, in contrast to surfaces, the problem of conjugacy of Morse functions on manifolds of dimension at least $3$ is much more difficult due to their more complicated surgery description (cf. \cite{Prishlyak-3-manifolds}). One approach is to equip Reeb graphs with labels on edges, which correspond to diffeomorphism type of submanifolds corresponding to them. An isomorphism of Reeb graphs which preserves these labels is the first necessary condition for conjugacy of functions. In particular, for orientable $3$-manifolds regular level sets are orientable surfaces, so we equip edges only with non-negative integers corresponding to their genera. Thus the problem is how many strong equivalence classes of Reeb epimorphisms of simple Morse functions with fixed Reeb graph with labels are there and how many of them correspond to a single conjugacy class of these functions.
	Especially, we ask whether there is a finite number of these classes.
	
	It should be noted that there are some results regarding genera of regular level sets of Morse functions on closed orientable $3$-manifolds. For example, O. Saeki \cite[Theorem 6.5]{Saeki_sphere_fibers} showed that there is a Morse function $f\colon M \to \R$ with maximum genus among labels in $\reeb(f)$ at most one if and only if $M$ is the connected sum of copies of $\es^3$, $\es^2\times \es^1$ and lens spaces. Moreover, N. Kitazawa \cite[Theorem 1]{Kitazawa} described a realization of such a labelled graph as the Reeb graph of a Morse function on a $3$-manifold. 
 	
 	\vskip 8pt
	3. The structure of the set of homomorphisms $\operatorname{Hom}(G,F_r)$ for $G$ finitely generated group has been intensively studied by several authors by use of the Makanin--Razborov diagrams theory (cf. \cite{Makanin_eqs_in_free_groups,Razborov}). It led to the solution of Tarski problem of the existence of solution to a system of equation in a free finitely generated group and culminated in the Sela's works (cf.~\cite{Sela} and more recent articles of this author; see also \cite{Bestvina} for a nice and relatively uncomplicated introduction to this theory). 
	The Makanin--Razborov diagram of a group $G$ consists of its quotients and quotient maps in such a way that every homomorphism $\varphi\colon G\to F_r$ is \emph{M--R factorized} through
	some branch of quotients $G\xrightarrow{q_1} L_1 \xrightarrow{q_2} \ldots \xrightarrow{q_k} L_k$, where $L_k$ is a free group. It means that there is an element $\psi\in {\rm Hom}(L_k,F_r)$ and modular automorphisms $\alpha \in {\rm Mod}(G)$ and $\alpha_i \in {\rm Mod}(L_i)$ for $1\leq i<k$ such that $\varphi = \psi \circ q_{k} \circ \alpha_{k-1} \circ \ldots \circ \alpha_1 \circ q_1 \circ \alpha$. 
	The equivalence and strong equivalence relations in $\operatorname{Hom}(G,F_r)$ cannot be derived from Makanin--Razborov diagrams in general. However, these notions are closely related  for groups with branches of length $1$ in their M--R diagrams, e.g. for surface groups.
	
	The computation of the set $\epi(\pi_1(M),F_r)$ of epimorphisms $\pi_1(M) \to F_r$ and the set of their strong equivalence classes $\epi(\pi_1(M),F_r)/_\simeq$ may also rely on the calculations of framed cobordism classes of independent systems of hypersurfaces. 
	By Theorem \ref{theorem:epimorphism_is_induced_by_regular_and_independent_system} there is a bijection $\epi(\pi_1(M), F_r) \cong \H^{fr}_r (M)$ and $\overline{\overline{\Theta}} \colon \H^{fr}_r(M)/_{\diffd(M)} \to \epi(\pi_1(M),F_r)/_\simeq$ is surjective. Moreover, in some cases it is bijective e.g. for hyperbolic manifolds.
	More generally, there are bijections
	$$
	{\rm Hom}(\pi_1(M),F_r) \cong [M,\bigvee^r \es^1] \cong \Omega^{fr,1}_r(M)
	$$
	between the sets of all homomorphisms $\pi_1(M)\to F_r$, homotopy classes of maps $M \to \bigvee^r \es^1$ and the set $\Omega^{fr,1}_r(M)$ of all framed cobordism classes of systems of hypersurfaces of size $r$ in $M$ (the upper index $1$ denotes the codimension $1$ of submanifolds). Consequently, any description of the latter set would give information about the corresponding set of homomorphisms.


\subsubsection*{Acknowledgements}
We would like to thank the reviewer whose comments and remarks improved the text of the work.


\address

\end{document}